\documentclass[a4paper,11pt]{article}

\usepackage{graphicx}
\usepackage{amsmath}
\usepackage{amsfonts}
\usepackage{titlesec}
\usepackage{amsthm}
\usepackage{amssymb}
\usepackage{geometry}
\usepackage{color}
\usepackage{accents}
\usepackage{epic}
\usepackage{caption}
\usepackage{subcaption}

\usepackage{algorithm2e}

\newenvironment{algo}{%
  \algorithm
}{%
  \endalgorithm
}

\newcommand{\ms}{\mbox{\rm \tiny MS}}
\newcommand{\rb}{\mbox{\rm \tiny RB}}

\newcommand{\Vmsk}{V_{H,k}^{\ms}}

\newcommand{\VmsRB}{V_{H,k}^{\rb}}

\newcommand{\WKzRB}{\mathcal{W}_{K,z}^{\rb}}

\newcommand{\WzRB}{W_{z}^{\rb}}
\newcommand{\VzRB}{V_{z}^{\rb}}

\newcommand{\uRB}{u_H^{\rb}}

\newcommand{\pRB}{p_H^{\rb}}
\newcommand{\pRBn}{p_H^{\rb,(n)}}
\newcommand{\pHn}{p_H^{\mbox{\tiny \rm c},(n)}}
\newcommand{\pH}{p_H^{\mbox{\tiny \rm c}}}

\newcommand{\deltaRBnn}{\delta_H^{\rb,(n+1)}}

\newcommand{\eRB}{e_H^{\rb}}
\newcommand{\eH}{e_H^{\mbox{\tiny \rm c}}}

\newcommand{\sym}{ \mbox{\tiny \rm sym} }
\newcommand{\train}{ \mbox{\tiny \rm train} }

\newcommand{\R}{\mathbb{R}}
\newcommand{\T}{\mathcal{T}}
\newcommand{\N}{\mathbb{N}}

\newcommand{\eps}{\varepsilon}

\newcommand{\aeps}{a^\eps}

\newcommand{\ueps}{u^\eps}

\newcommand{\beps}{b^{\eps}}
\newcommand{\bmu}{{\boldsymbol\mu}}

\newcommand{\quotes}[1]{``#1''}

\definecolor{dark-green}{rgb}{0.0,0.4,0.0}
\definecolor{greenish}{rgb}{0.0,0.7,0.3}

\newtheorem{theorem}{Theorem}[section]

\newtheorem{proposition}[theorem]{Proposition}

\theoremstyle{definition}
\newtheorem{definition}[theorem]{Definition}
\newtheorem{remark}[theorem]{Remark}
\newtheorem{notation}[theorem]{Notation} %
\usepackage{framed}

\newenvironment{fshaded}{%
\MakeFramed {\FrameRestore}}%
{\endMakeFramed}

\newtheorem{cstep}{Step}

\newenvironment{step}[1][]{\definecolor{shadecolor}{rgb}{.93,.93,1}%
\definecolor{framecolor}{rgb}{.95,.95,1}%
\begin{fshaded}\begin{cstep}[#1]\mbox{}\\\noindent}{\end{cstep}\end{fshaded}\medskip}

\title{A reduced basis localized orthogonal decomposition}

\begin{document}
\maketitle

\begin{center}
{\large Assyr Abdulle\footnote[1]{ANMC, Section de Math\'{e}matiques, \'{E}cole polytechnique f\'{e}d\'{e}rale de Lausanne, 1015 Lausanne, Switzerland, Assyr.Abdulle@epfl.ch} and Patrick Henning\footnote[2]{Institut f\"ur Numerische und Angewandte Mathematik, Westf\"alische Wilhelms-Universit\"at M\"unster, Einsteinstr. 62, D-48149 M\"unster, Germany, Patrick.Henning@wwu.de}}\\[2em]
\end{center}

\renewcommand{\thefootnote}{\fnsymbol{footnote}}
\renewcommand{\thefootnote}{\arabic{footnote}}

\begin{abstract}
In this work we combine the framework of the Reduced Basis method (RB) with the framework of the Localized Orthogonal Decomposition (LOD) in order to solve parametrized elliptic multiscale problems. The idea of the LOD is to split a high dimensional Finite Element space into  a low dimensional space with comparably good approximation properties and a remainder space with negligible information. The low dimensional space is spanned by locally supported basis functions associated with the node of a coarse mesh obtained by solving decoupled local problems. However, for parameter dependent multiscale problems, the local basis has to be computed repeatedly for each choice of the parameter. To overcome this issue, we propose an RB approach to compute in an \quotes{offline} stage LOD for suitable representative parameters. The online solution of the multiscale problems  can then be obtained in a coarse space (thanks to the LOD decomposition) and for an arbitrary value of the parameters (thanks to a suitable \quotes{interpolation} of the selected RB). The online RB-LOD has a basis with local support and leads to sparse systems. Applications of the strategy to both linear and nonlinear problems are given.
\end{abstract}



\section{Introduction}
\label{section:introduction}

In this paper, we consider parametrized linear elliptic multiscale problems, i.e. we are interested in finding the parameter-dependent solution $\ueps(\cdot\hspace{2pt};\cdot)$ of an equation
\begin{align}
\nonumber\label{equation-strong} - \nabla \cdot \left( \aeps(x;{\bmu}) \nabla \ueps(x;{\bmu}) \right) &= f(x;{\bmu}) \qquad \mbox{in } \Omega, \\
\ueps(x;{\bmu}) &= 0 \hspace{38pt} \mbox{on } \partial \Omega.
\end{align}
Here, ${\bmu}=(\mu_1,\ldots,\mu_P)$ denotes a parameter vector. It is an element of a multidimensional parameter set $\mathcal{D} \subset \R^{P}$, where $P \in \mathbb{N}$. The parameter-dependent coefficient matrix $\aeps(x;{\bmu})$ is assumed to be a {\it multiscale coefficient}. It exhibits a continuum of different scales, where the finest scale is very small compared to the size of computational domain $\Omega$. In particular $\aeps(x;{\bmu})$ shows very fast variations that need to be resolved with an extremely fine computational grid. The order of the fines scale in our problem is characterized by the abstract quantity $0<\eps\ll 1$. However, we do not need to assign a specific value to $\eps$. Due to the requirement that all scales of $\aeps(\cdot;{\bmu})$ need to be resolved with a computational grid, the problem cannot be tackled by standard methods (such as classical finite element methods) since the computational complexity would become prohibitively large.
Hence we are interested in finding a way to decrease the computational complexity and to distribute the load on several CPUs by introducing fully decoupled local subproblems. Furthermore, we want to avoid recomputing local subproblems for every new parameter ${\bmu}$. We are thus looking for (a small number of) representative parameters for which accurate local problems and bases are computed and that allow for fast computations for every new parameter ${\bmu}$.

Parameter-dependent multiscale problems can for instance arise in applications from material sciences, geophysics or hydrology. More specific examples are the prediction of global strain or elasticity properties of fiber reinforced composite materials, where the parameters can describe different constellations for the microscopic fibers that are embedded in the main material (e.g. their form or density). Another example is the flow in porous media where different permeability configurations can be parametrized. For such cases the coefficient $\aeps(\cdot;{\bmu})$ and the source term $f(\cdot;{\bmu})$ can both depend on a large number of parameters ${\bmu}$. It is therefore of strong interest to construct methods that combine the features of a multiscale method (to treat the rapid variations in the coefficients) with a reduced basis approach (to treat the dependency on a large set of parameters).

There are numerous different methods that are designed to treat the classical (parameter-free) multiscale problems (cf. \cite{Abd05b,Abd09a,BaL11,BaL11b,BFH97,EE03,GGS12,Glo06,Glo11,HoW97,HFM98,Hug95,HWC99,MaP14,Mal11,OwZ11,OZB13} and the references therein). In this paper we focus on the localized orthogonal decomposition (LOD) introduced in \cite{MaP14}.
To handle parameter dependency in an efficient way we will build on the  reduced basis (RB) approach (cf. \cite{GNV07,MMO00,PaR07,PRV02,PRV02B,RHP08}). The reduced basis method is a model order reduction technique that we describe at the end of this section when we describe the idea of the reduced basis localized orthogonal decomposition approach (RB-LOD).

Despite the large number of results on multiscale methods and reduced basis approaches there are only few works which combine both features. In the context of
periodic homogenization this was first studied by Boyaval \cite{Boy08,Boy09} and extended for more general numerical homogenization problems in \cite{AbB12,AbB13,AbB14,ABV14}, where the
{\it reduced basis finite element heterogeneous multiscale method} (RB-FE-HMM) has been introduced.
The RB-FE-HMM was originally designed to reduce the computational complexity of the classical Heterogenous Multiscale Method \cite{EE03} by interpreting the location of a cell problem as a parameter (which is equivalent to the dependency on the coarse variable). With that strategy, precomputed solutions from other cell problems can be used to construct reduced basis solution spaces for new cell problems. This method also generalizes to additional parameter dependencies such as in (\ref{equation-strong}). A similar approach which also fits into the HMM framework was presented in \cite{OhS12}, where the focus is on optimization problems that are constrained by a parameterized multiscale problem. A combination of the RB framework with the multiscale finite element method (MsFEM, see \cite{HWC99}) was proposed by Nguyen in \cite{Ngu08}, model reduction techniques for the MsFEM have also been developed in \cite{EGH13}.
Finally, we mention the approach of the localized reduced basis multiscale method (LRBMS) proposed in \cite{KOH11,AHK12} and further developed in \cite{OhS14,KFH14}. The main idea of the method is to localize global solutions (that were determined for a set of parameters) to the elements of a coarse grid. The localization can be simply obtained by truncation and hence the localized solutions can be used as basis functions in a global discontinuous Galerkin approach. 

The Reduced Basis framework can be combined with most of the multiscale methods mentioned in the introduction. In this paper we chose the LOD because it has some attractive features compared to other approaches that also aim to solve multiscale problems without scale separation. For instance, even though an RB Multiscale Finite Element Method (RB-MsFEM, cf. \cite{Ngu08}) is computationally less expensive, it suffers from the the constrained that it requires strong structural assumptions on $\aeps(\cdot,\bmu)$ (such as local periodicity). Efficient and reliable methods that do not suffer from such a constrained are for instance the approaches proposed by Owhadi and Zhang \cite{OwZ11} or Babuska and Lipton \cite{BaL11}. The approach by Owhadi and Zhang exploits a so called transfer property (comparable to a harmonic coordinate transformation) and requires to solve local problems in patches of sizes of order $\sqrt{H} |\log(H)|$ to guarantee an optimal linear convergence rate in $H$ for the $H^1$-error. Compared to that, the LOD only requires patches with a diameter of order $H |\log(H)|$. The method of Babuska and Lipton \cite{BaL11} has a different structure and even smaller patches can be picked. Here optimal local approximation spaces are constructed. However, this local construction requires to incorporate the source term $f(\cdot,\bmu)$ by solving additional local problems of the structure $- \nabla \cdot (\aeps(\cdot,\bmu) \nabla v^{\eps}(\cdot,\bmu)) = f(\cdot,\bmu)$. In order to account for this in the RB-framework, an affine decomposition of $f(\cdot,\bmu)$ is required. Furthermore, the costs for the offline phase are increased. Compared to that, the LOD-approach involves local spaces that are independent of $f$, without suffering from a reduction of the convergence rates.

In this paper we introduce the reduced basis local orthogonal decomposition (RB-LOD). We briefly summarize the main ideas.
Consider a coarse triangulation $\T_H$ and a corresponding set of coarse nodes $\mathcal{N}_H$. For any fixed (i.e. parameter independent) coefficient $\aeps$ the LOD is designed to construct a set of (locally supported) multiscale basis functions $\Phi_z^{\ms}$ (each of them associated with a single coarse node $z\in \mathcal{N}_H$) so that the discrete space that is spanned by these basis functions yields the classical convergence rates in $H$.
The functions $\Phi_z^{\ms}$ are obtained from  the solution of a local finite element problem (in a local space that resolves the microstructure). The coarse triangulation $\T_H$ does not need to resolve the microstructure and can hence be low dimensional. However if the coefficient $\aeps(\cdot;{\bmu})$ is parameter-dependent then $\Phi_z^{\ms}({\bmu})$ is parameter-dependent as well and needs to be recomputed again for any new parameter. To overcome this drawback we apply the reduced basis method together with a Greedy search algorithm to identify a set of parameters for which we compute $\Phi_z^{\ms}$. These solutions can be used to construct affine (reduced basis) spaces $\VzRB$, for each node $z \in \mathcal{N}_H$.

The computation of the spaces $\VzRB$ takes place in an offline phase (i.e. it is a preprocessing step). The functions in $\VzRB$ are only locally supported in a small patch around the node $z$.
Once constructed, these reduced basis (multiscale) spaces can then be used in an online phase to obtain a solution of the problem for any new parameter in a coarse reduced RB-LOD space.
The strategy proposed here allows to construct a localized reduced basis, since different parameter sets can be used locally.
Furthermore, we do not need to assume that there exists an affine decomposition for $f(\cdot;{\bmu})$.

We now introduce the assumptions that we use throughout the paper. The physical domain  $\Omega\subset\mathbb{R}^{d}$, for $d=1,2,3$ will be assumed to be a bounded Lipschitz domain with a piecewise polygonal boundary.

In order to guarantee well-posedness of the problem we assume the following.
\begin{itemize}
\item[(A1)] for every parameter $\bmu\in \mathcal{D}$ we have $f(\cdot\hspace{2pt};\bmu) \in L^2(\Omega)$; furthermore 
we assume 
that there exists $C\in \R$ such that
$\|f\|_{L^2(\Omega,L^{\infty}(\mathcal{D}))} \le C$;
\item[(A2)] the matrix-valued parameter-dependent functions $\aeps(\cdot\hspace{2pt};{\bmu})\in [L^\infty(\Omega)]^{d\times d}_{\sym}$ have uniform spectral bounds, i.e. there exist real numbers $0<\alpha\le\beta$ such that for all $\bmu \in \mathcal{D}$ and almost every $x \in \Omega$
\begin{align*}
\forall \xi \in \R^d: \qquad \alpha |\xi|^2 \le \aeps(x;\bmu) \xi \cdot \xi \le \beta |\xi|^2.
\end{align*}
\end{itemize}
To make sure that the reduced-basis method can be efficiently implemented, we require another assumption.
\begin{itemize}
\item[(A3)] The parameter set $\mathcal{D}$ is compact in $\R^{P}$. Furthermore, there exists a finite index set $\mathcal{Q} \subset \mathbb{N}$, measurable parameter-independent functions $\aeps_q \in L^{\infty}(\Omega)$ (for $q\in \mathcal{Q}$) and measurable functions $\Theta_q : \mathcal{D} \rightarrow \R$ (for $q\in \mathcal{Q}$) such that $\aeps(x;{\bmu})$ has the affine representation
\begin{align*}
\aeps(x;{\bmu}) = \sum_{q \in \mathcal{Q}} \Theta_q({\bmu}) \aeps_q(x). 
\end{align*}
\end{itemize}
We notice that if (A3) does not hold, reduced-basis techniques can still be used by relying on the so-called empirical interpolation method that allows to approximate a tensor $\aeps(x;{\bmu})$ by a decomposition similar to (A3) (cf. \cite{BMN04}). Such techniques could also be used in the present paper but as this is not the main focus of the paper we rather assume the decomposition (A3) to be already at hand.

\begin{remark}[Limitations of the method]
Assumption (A3) gives an insight in the limitations of the RB-LOD that we propose in this paper. If the 
coefficient $a^{\epsilon}(x;\bmu)$ is not smooth with respect to the parameter $\bmu$, an application of an empirical interpolation transformation might not be possible or at least results in a large number of terms in the affine representation 
$\aeps(x;{\bmu}) = \sum_{q \in \mathcal{Q}} \Theta_q({\bmu}) \aeps_q(x).$
However, if $\mathcal{Q}$ is large, the RB-LOD approach might not pay off, since the offline stage gets computationally expensive and requires to store a lot of pre-computed data functions. In particular, the computation of local Riesz representatives (as required by the method, see Step 3 below) becomes very costly. Hence, $\aeps$ should smoothly depend on ${\bmu}$ so that $\mathcal{Q}$ remains small. Similarly, geometry related parameter-dependencies of $a^{\epsilon}(x;\bmu)$ are typically difficult to handle, since they also result in a large number of terms in the affine representation.

On the other hand, if $\aeps(x;{\bmu})$ models a stochastic medium with a smooth dependency on the stochastic variable, the Karhunen-Lo\`{e}ve expansion exhibits an exponential convergence concerning truncation in the number of terms (cf. \cite{ScT06}). Consequently, small values of ${\mathcal{Q}}$ can be expected in this case. Other stochastic applications are elliptic partial differential equations with small uncertainties (cf. \cite{GNP14}). Here, the affine representation of $\aeps(x;\bmu)$ is given as some deterministic $a^0(x)$ plus a small (basically deterministic) perturbation.
\end{remark}

Under assumptions (A1)-(A2), for any given parameter $\bmu \in \mathcal{D}$ there exists a unique weak solution $\ueps(\cdot\hspace{2pt};{\bmu}) \in H^1_0(\Omega)$ of (\ref{equation-strong}) with
\begin{align}
\label{equation-weak}
\left( \aeps(\cdot\hspace{2pt};{\bmu}) \nabla \ueps(\cdot\hspace{2pt};{\bmu}), \nabla v \right)_{L^2(\Omega)} &= \left( f(\cdot\hspace{2pt};\bmu), v \right)_{L^2(\Omega)} \qquad \mbox{for all } v \in H^1_0(\Omega).
\end{align}
For simplicity we use the notation
\begin{align}
\label{definition:problem:bilinearform}\beps(v,w;{\bmu}):=\left( \aeps(\cdot\hspace{2pt};{\bmu}) \nabla v, \nabla w \right)_{L^2(\Omega)}  \qquad \mbox{for } v,w \in H^1_0(\Omega)
\end{align}
and, furthermore, for every subdomain $\omega \subset \Omega$ we denote the local energy norm by
\begin{align*}
\| v \|_{\mathcal{E}(\omega)}^{{\bmu}} := \| \aeps(\cdot\hspace{2pt};{\bmu})^{1/2} \nabla v \|_{L^2(\omega)} \qquad \mbox{for } v \in H^1(\omega).
\end{align*}

{The paper is organized as follows: In Section \ref{subsection-LOD} we recall the definition and the main features of the classical Localized Orthogonal Decomposition. In Section \ref{section-rb-decomposition} we combine this approach with the Reduced Basis method. In particular we present algorithms which give a step-by-step procedure for how to implement the combined method. In Section \ref{section-numerical-experiments} we state the results of two numerical experiments. The first experiment involves a parametrized linear problem, whereas the second experiment demonstrates the applicability of the method to nonlinear problems.}

\section{Localized Orthogonal Decomposition}
\label{subsection-LOD}

We start by introducing a general space discretization of the problem \eqref{equation-strong}, that is based on the framework of the Localized Orthogonal Decomposition (LOD, cf. \cite{MaP14b,MaP14,HeM14,HeP13,HMP15,HMP14b}). Galerkin approximations in standard finite element spaces are known to suffer from pre-asymptotic effects for multiscale problems. More precisely, if $u_H$ denotes the FEM Galerkin approximation in a P1 FEM space $V_H$, the optimal convergence order (under the assumption of sufficient regularity) will be $\mathcal{O}(C(\eps^{-1}) H)$, where $C(\eps^{-1})$ depends on the speed of the oscillations (comparable to the size of the derivative of $\aeps$ - if it exists, e.g. $C(\eps^{-1})={\cal O}(\eps^{-1})$ in periodic homogenization). Hence, $H$ must be smaller than $C(\eps^{-1})$ so that the method enters the classical asymptotic regime of linear convergence in $H$. The idea of the LOD method is to construct subspaces of $H^1_0(\Omega)$, in which we obtain convergence without pre-asymptotic effects, i.e. the convergence rates do not depend on $C(\eps^{-1})$. Consequently, low dimensional subspaces can be used to obtain highly accurate results. Below we describe the construction of the LOD spaces that we denote by $\Vmsk$.

To discretize the problem, we require two families of computational grids on $\Omega$: a family of fine meshes $\T_h$ and a family of of coarse mesh $\T_H$. Both families are assumed to consist of  conforming and shape regular simplicial elements and we denote by $H$ or $h$ the maximum diameter of an element of $\T_H$ or $\T_h$, respectively.
We furthermore assume that  $\T_h$ originates from a regular mesh refinement of $\T_H$. The size of the macro mesh $\T_H$ is not constrained by the microstructure of the solution
and linear convergence towards  $\ueps(\cdot;{\bmu})$  will be obtained in the macro mesh size $H$.
The mesh $\T_h$ is assumed to be fine enough to resolve the rapid variations of the coefficients $\aeps_q$ for $q\in \mathcal{Q}$. For $\T=\T_H,\T_h$ and $\omega$ a subset of coarse or fine elements we denote by $P_1(\omega,\T)$ the space of continuous functions on $\omega$ that are linear in each $K\in\T$.

We then define $V_H:=P_1(\Omega,\T_H)\cap H^1_0(\Omega)$ and the space $V_h$ is defined accordingly. 
The set of interior coarse nodes (interior Lagrange points) of $\T_H$ will be denoted by $\mathcal{N}_H$.
Furthermore, we denote $N:=|\mathcal{N}_H|$ the number of nodes. For a given node $z \in \mathcal{N}_H$ the corresponding coarse nodal basis function is $\Phi_z \in V_H$ (i.e. $\Phi_z(z)=1$ and $\Phi_z(y)=0$ for all $y \in \mathcal{N}_H \setminus \{ z\}$). We set $\omega_z:=\text{supp}(\Phi_z)$ and by $N_z$ we denote the number of coarse elements in $\omega_z$, i.e. $N_z := |\{ K \in \T_H| \hspace{2pt} K \subset \omega_z\}|$.

In the next step, we define a \quotes{remainder space} or \quotes{detailed space} $W_h$ that contains functions with a small $L^2$-norm. To define this space we make use of the following Cl\'ement-type quasi-interpolation operator $I_H$ that was introduced in \cite{Car99}. We define
\begin{align}
\label{def-weighted-clement} I_H : H^1_0(\Omega) \rightarrow V_H,\quad v\mapsto I_H(v):= 
\sum_{z \in \mathcal{N}_{H}}
v_z \Phi_z \quad \text{with }v_z := \frac{(v,\Phi_z)_{L^2(\Omega)}}{(1,\Phi_z)_{L^2(\Omega)}}
\end{align}
and set $W_h := \{ v_h \in V_h | \hspace{2pt} I_H(v_h) = 0 \}$.

\begin{remark}\label{remark-on-L2-projection}
The operator $I_H$ is closely related to the $L^2$-projection $P_{L^2} : V_h \rightarrow V_H$. In particular it holds
$$(I_H\vert_{V_H})^{-1} \circ I_H = P_{L^2}.$$
For details, we refer to \cite[Remark 3.8]{EGH15} (see also \cite{MaP14b,Car99,MaP14}).
\end{remark}

As was shown in \cite{MaP14} the $\beps(\cdot,\cdot;{\bmu})$-orthogonal complement of $W_h$ in $V_h$ has  good $H^1$-approximation properties with respect to the exact solution $\ueps(\cdot\hspace{2pt};\bmu)$. Practically, it is very expensive to compute this orthogonal complement exactly, however it can be accurately approximated by the following cheap localization strategy based on an affine decomposition of the $\beps(\cdot,\cdot;{\bmu})$-orthogonal projection operator from $V_h$ in $W_h$.

\begin{figure}[htb]
  \centering
  \begin{subfigure}{.47\textwidth}
    \centering
    \includegraphics[width=\textwidth]{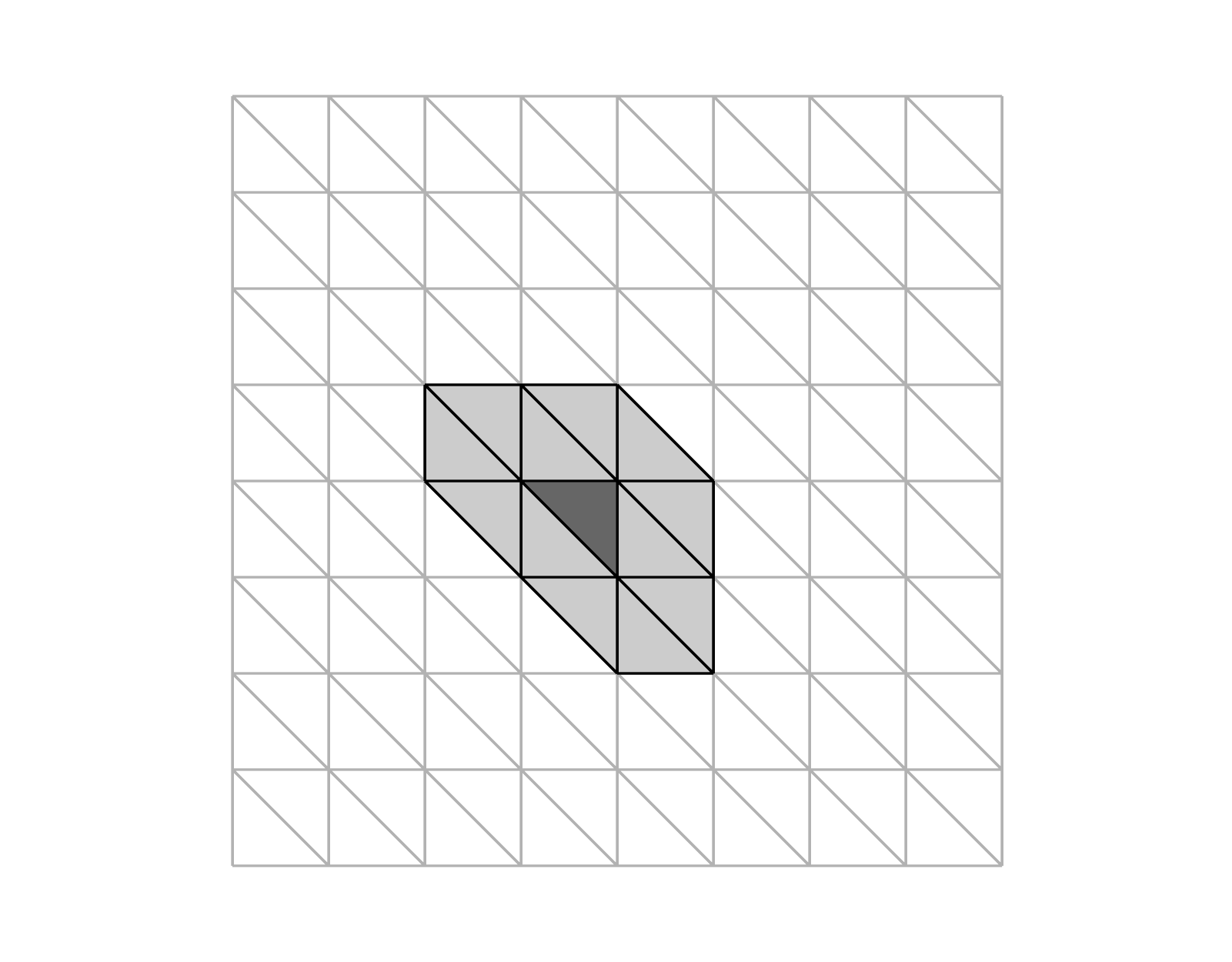}
    \caption{The patch $U_1(K)$ for a given $K\in \T_H$.}
  \end{subfigure}
  \hspace{1em}
  \begin{subfigure}{.47\textwidth}
    \centering
    \includegraphics[width=\textwidth]{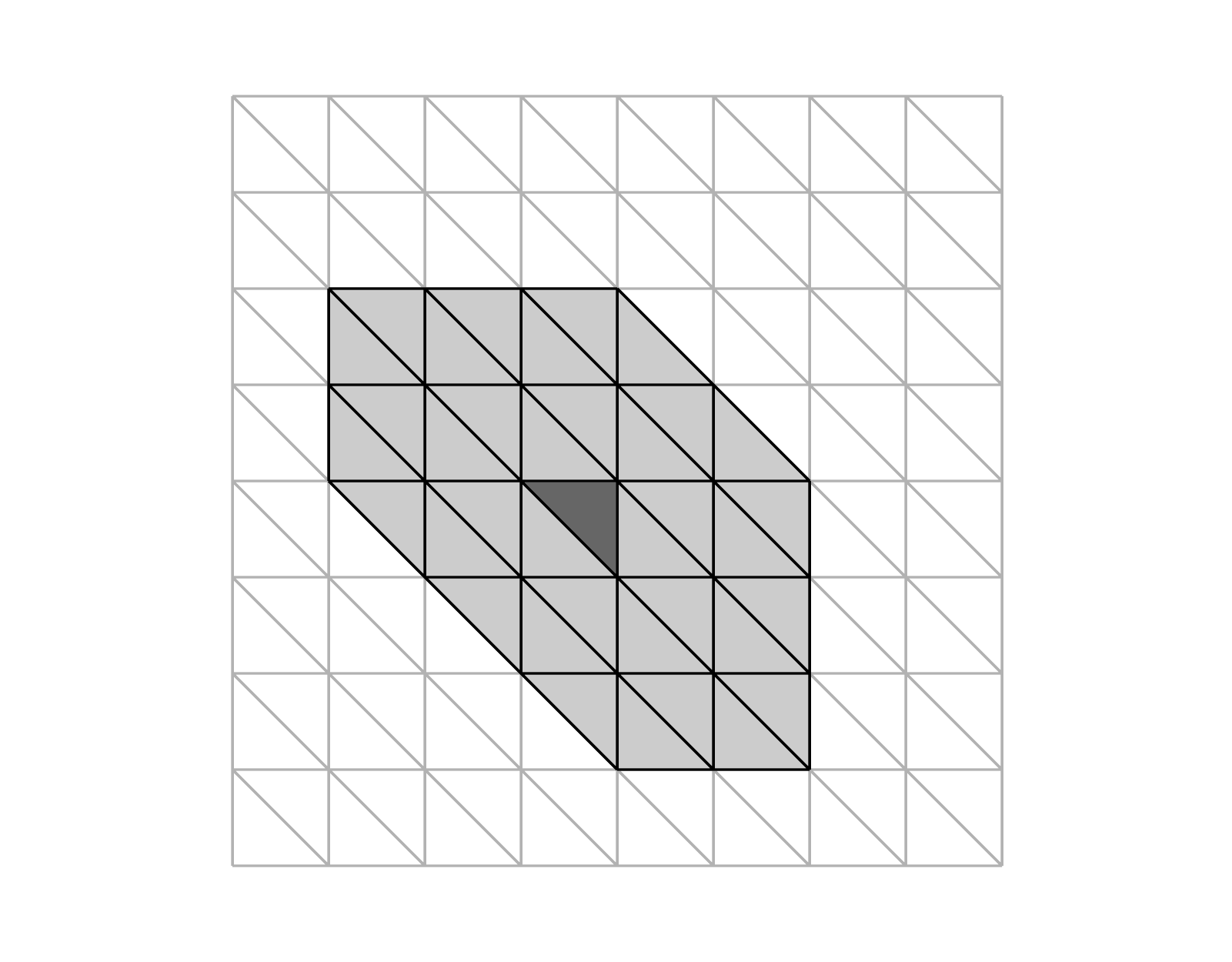}
    \caption{The patch $U_2(K)$ for a given $K\in \T_H$.}
  \end{subfigure}
  \caption{Illustration of the 
  patches $U_k(K)$ defined in \eqref{def-patch-U-k}. The dark gray part depicts the coarse element $K\in\T_H$. The light and the dark gray part together depict the patch $U_k(K)$.}
  \label{fig:patch}
\end{figure}

\begin{definition}[Localized orthogonal complement]
\label{definition:localized:ms:space}
Let us fix $\bmu \in \mathcal{D}$. First, for $k\in \mathbb{N}$ and $K \in \T_H$ we define the patch $U_k(K)$ iteratively by
\begin{equation}\label{def-patch-U-k}
    \begin{aligned}
      U_0(K) & := K, \quad \mbox{and} \quad
      U_k(K) & := \cup\{T\in \T_H\;\vert\; \overline{T}\cap \overline{U_{k-1}(K)}\neq\emptyset\}\quad k=1,2,\ldots
    \end{aligned}
\end{equation}
See Figure \ref{fig:patch} for an illustration.
Then, define the localized remainder space by
\begin{align}
\label{local-kernel-space}
W_h(U_k(K)):=\{ w_h \in W_h| \hspace{2pt} w_h=0 \enspace \mbox{in } \Omega \setminus U_k(K) \}.
\end{align}
Consider a nodal basis function $\Phi_z$. For $k\in \mathbb{N}_{>0},$ $K \in \T_H$ with $K \subset \omega_z$, we 
define $Q_{h,k}^{K}(\Phi_z;{\bmu}) \in {W}_h(U_k(K))$ as the solution of
\begin{align}
\label{local-corrector-problem}\int_{U_k(K)} \aeps(\cdot \hspace{2pt}; {\bmu} ) \nabla Q_{h,k}^{K}(\Phi_z;{\bmu})\cdot \nabla w_h = - \int_K \aeps(\cdot \hspace{2pt}; {\bmu} ) \nabla \Phi_z \cdot \nabla w_h \qquad \mbox{for all } w_h \in W_h(U_k(K)).
\end{align}
Then define $Q_{h,k}(\cdot \hspace{2pt};{\bmu}) : V_H \rightarrow W_h$ by
\begin{align}
\label{global-corrector}Q_{h,k}(\Phi_z;{\bmu}):=\underset{K \subset \omega_z}{\sum_{K\in \T_H}}
Q_{h,k}^{K}(\Phi_z;{\bmu}).
\end{align}
and set
\begin{align}
\label{localized-ms-space}
\Vmsk({\bmu}):=\mbox{\rm span} \{ \Phi_z + Q_{h,k}(\Phi_z;{\bmu})| z \in \mathcal{N}_H \hspace{2pt} \}
\end{align}
to be the approximation of the $\beps(\cdot,\cdot;{\bmu})$-orthogonal complement of $W_h$ in $V_h$. Observe that we get the exact orthogonal complement for the case that $U_k(K)=\Omega$ for all $K\in \T_H$.
\end{definition}

Observe that problem (\ref{local-corrector-problem}) is a constrained problem since solution and test functions need to be in the kernel of the interpolation operator $I_H$. This is practically realized by introducing a corresponding Lagrange multiplier and problem (\ref{local-corrector-problem}) becomes hence a saddle point problem
of the structure: find $(\vec{\mathbf{w}},\vec{\boldsymbol{\lambda}}) \in \mathbb{R}^{N_{h,U_k(K)}} \times \mathbb{R}^{N_{H,U_k(K)}}$ with
\begin{align*}
\mathbf{S}_h \vec{\mathbf{w}} + \mathbf{C}_h^{\top} \vec{\boldsymbol{\lambda}} &= \vec{\mathbf{r}}\\
\mathbf{C}_h \vec{\mathbf{w}} &= 0.
\end{align*} 
Here, $N_{h,U_k(K)}$ denotes the number of fine grid nodes in $U_k(K)$ and $N_{H,U_k(K)}$ the number of coarse grid nodes in $U_k(K)$. The solution vector $\vec{\mathbf{w}}$ is the coefficient vector that defines $Q_{h,k}^{K}(\Phi_z;{\bmu})$ and $\vec{\boldsymbol{\lambda}}$ is the corresponding Lagrange multipler for the constraint $I_H(Q_{h,k}^{K}(\Phi_z;{\bmu}))=0$. The matrix $\mathbf{S}_h$ denotes the standard P1-FEM stiffness matrix associated with the left hand side of (\ref{local-corrector-problem}), $\mathbf{C}_h$ is the localized algebraic version of $I_H$ and $\vec{\mathbf{r}}$ is the load vector associated with the right hand side of (\ref{local-corrector-problem}). Since $N_{H,U_k(K)}$ is a small number, it is not necessary to solve the system with an iterative solver. Instead it can be directly solved by computing the Schur complement and inverting it. The costs for solving all local problems associated with a patch $U_k(K)$ are hence of order $\mathcal{O}(N_{H,U(K)} N_{h,U(K)})$ (provided that algebraic solvers with a linear complexity are used). If $\mathcal{T}_H$ and $\mathcal{T}_h$ are quasi-uniform and if $k\simeq |\log(H)|$ (as suggested by Proposition \ref{prop-conv-lod} below)
we have $N_{H,U(K)}\simeq  |\log(H)|^d$ and $N_{h,U(K)}\simeq (H |\log(H)| / h)^d$. Hence, the total cost for solving all local problems associated with a patch are of order $(H |\log(H)|^2 / h)^d$.

Also note that the assembly of $\Vmsk({\bmu})$ is parallelizable and cheap for small values of $k \in \N$ (then $U_k(K)$ is a small subdomain with a diameter of order $H$). In numerical experiments (cf. \cite{HeM14,HeP13,HMP15,HMP14}) it was demonstrated that $k=1,2,3$ is typically sufficient. The reason is the exponential decay of $Q_{h,k}^{K}(v_H;{\bmu})$ outside of $K$. A quantification of this statement is given in Proposition \ref{prop-conv-lod} below.

In the following, we use the notation $a \lesssim b$, which stands for $a\leq Cb$, where $C$ is a constant that does not depend on the parameter $\bmu$, $k$, the mesh sizes $H$ and $h$ or the rapid oscillations in $\aeps$ (i.e. $\eps$).

The proof of the following proposition follows the line of \cite{HeM14}.
\begin{proposition}
\label{prop-conv-lod}
Assume that (A1)-(A2) hold and let $\Vmsk({\bmu})$ be given by (\ref{localized-ms-space}). We consider 
the following problem: find $u^{\ms}(\cdot\hspace{2pt};\bmu) \in \Vmsk({\bmu})$ such that
\begin{align}
\label{classical-LOD-solution-eq}\beps( u^{\ms}(\cdot\hspace{2pt};{\bmu}), v; {\bmu}) &= \left( f(\cdot\hspace{2pt};\bmu), v \right)_{L^2(\Omega)} \qquad \mbox{for all } v \in \Vmsk({\bmu}).
\end{align}
Then, there exists a generic constant $C_g$ (i.e. independent of $H$, $h$ and $\eps$, {but possibly depending on the contrast $\beta/\alpha$}) such that if $k \ge C_{g} |\log(H)|$ it holds
\begin{align*}
\| u^{\ms}(\cdot\hspace{2pt};{\bmu}) - u_h(\cdot\hspace{2pt};{\bmu}) \|_{H^1(\Omega)} \lesssim H.
\end{align*}
Here, the fine scale reference $u_h(\cdot\hspace{2pt};\bmu) \in V_h$ is defined as the solution of 
\begin{align}
\label{reference-solution}\beps( u_h(\cdot\hspace{2pt};{\bmu}), v_h; {\bmu}) &= \left( f(\cdot\hspace{2pt};\bmu), v_h \right)_{L^2(\Omega)} \qquad \mbox{for all } v \in V_h.
\end{align}
In numerical experiments it can be observed that $C_{g}$ can be typically replaced by $1$ to find a suitable value for $k$ (cf. \cite{HeM14}).
\end{proposition}

Proposition \ref{prop-conv-lod} states that the LOD approach preserves the linear order of convergence (for the $H^1$-error) of the classical Finite Element Method without pre-asymptotic effects, even in case of rough coefficients $a^\eps$.

\section{Reduced Basis Decomposition}
\label{section-rb-decomposition}

Assume that the multiscale space $\Vmsk({\bmu})$ shall be assembled for an arbitrary $\bmu$ from a large set of relevant parameters $\Xi^{\train} \subset \mathcal{D}$. Depending on the size of $\Xi^{\train}$ this can be prohibitively expensive. We therefore ask the question: is it possible to only select a {\it small} subset $\Xi^{\rb} \subset \Xi^{\train}$, assemble $\Vmsk({\bmu})$ {\it only} for $\bmu \in \Xi^{\rb}$ and then reuse (or combine) these results to quickly/cheaply find an approximation of $\Vmsk({\bmu_{new}})$ for any new ${\bmu_{new}} \in \Xi^{\train} \setminus \Xi^{\rb}$.

This can be achieved by making use of the framework of the Reduced Basis (RB) method. In the following, we will elaborate the approach in detail. The goal is to construct (affine) local Reduced Basis spaces $\WzRB$ for each coarse node $z \in \mathcal{N}_H$ and, given a $\bar{\bmu} \in \Xi^{\train}$, to select one element from each of these local spaces to span a global multiscale RB space $\VmsRB(\bar{\bmu})$. Any of the following steps is either categorized as {\it offline} or {\it online}. By {\it offline step} we mean, that the step should be considered as an \quotes{one time preprocessing step}. It is significantly more expensive than an online step but need to be computed only once. This step provides a selection of representative parameters and corresponding LOD that can be used for a new parameter in the \quotes{online step}. Every step that takes place after the preprocessing is finished, is what we call an online step. Online steps are quick and efficient. They can be performed a lot of times for a lot of different parameters without involving a considerable computational complexity. 

\subsection{Initialization}
\label{subsection-initialization}
To describe the RB-LOD procedure, we start with Step \ref{step-1}, where we make an initial selection for the training set $\Xi^{\train} \subset \mathcal{D}$.

\begin{step}[Offline - Choice of a training set]\label{step-1}$\\$
In the first step, we randomly choose a finite subset $\Xi^{\train}$ of the parameter set $\mathcal{D}$. We assume that the so called {\it training set} $\Xi^{\train}\subset\mathcal{D}$ is sufficiently large so that the method is stable. Practically, $\Xi^{\train}$ can be for instance determined with the Monte-Carlo method.
\end{step}

Now that the training set is determined, the main computation involves to find (possibly small) parameter subsets $\Xi_{z}^{\rb}$ (for each coarse node $z \in \mathcal{N}_H$, this parameter set might change) for which the localized orthogonal decomposition is performed according to Definition \ref{definition:localized:ms:space}. We start with a random initial parameter choice $\bmu_1$ as described in Step \ref{step-2}. Recall that $\Phi_z \in V_H$ denotes the coarse nodal basis function belonging to the node $z \in \mathcal{N}_H$.

\begin{step}[Offline - Initialization with starting parameter]\label{step-2}
\begin{algo}
\label{algorithm-initialization}
 \rule{0.8\textwidth}{.7pt} \\
Pick randomly ${\bmu}_1 \in \Xi^{\train}$. Set $\Xi_{z}^{\rb}:=\{ {\bmu}_1 \}$ for all $z \in \mathcal{N}_H$. 
 \rule{0.8\textwidth}{.7pt} \\
Algorithm: initialize( $\Xi_{z}^{\rb}$ )
 \rule{0.8\textwidth}{.7pt} \\
In parallel \ForEach{$K \in \T_H$}
{
\ForEach{$z \in \mathcal{N}_H$ \mbox{\rm with} $z \in \overline{K}$}
 {
   compute $Q_{h,k}^{K}(\Phi_z;{\bmu}_1) \in {W}_h(U_k(K))$ \mbox{\rm via (\ref{local-corrector-problem}).}\\
   Set $\WKzRB:=\mbox{\rm span}\{ Q_{h,k}^K(\Phi_z;{\bmu}_1) \}$ (local RB space).
 }
}
\rule{0.8\textwidth}{.7pt}
\end{algo}
\end{step}

After Step \ref{step-2}, we constructed a first (trivial) global RB multiscale space $\VmsRB(\bmu)$, where
\begin{align*}
\VmsRB(\bmu):=\mbox{\rm span}\{ \Phi_{z}^{\ms}({\bmu}_1)| \hspace{2pt} z \in \mathcal{N}_H \} \qquad \mbox{with } \quad \Phi_{z}^{\ms}({\bmu}_1) := \Phi_z + \underset{K \subset \omega_z}{\sum_{K\in \T_H}} Q_{h,k}^{K}(\Phi_z;{\bmu}_1).
\end{align*}
Note that $\VmsRB(\bmu)$ as above is just a preliminary space that can be used for any parameter $\bmu$, but typically without good approximation properties. Similarily, the spaces $\WKzRB$ (for $z\in \mathcal{N}_H$ and $K \in \T_H$ and $K\subset \omega_z$) are the preliminary ($1$-dimensional) local Reduced Basis spaces. In the next step, we want to update the RB sample sets $\Xi_{z}^{\rb}$ and consequently the local RB spaces $\WKzRB$. By \quotes{update} we mean that we want to add suitable elements ${\bmu}_{z,2}$ from the training set $\Xi^{\train}$ to the local RB parameter sets $\Xi_{z}^{\rb}$. This can be achieved using a Greedy search algorithm based on a posteriori error estimation. A general result is given in the next paragraph.

\subsection{Local a posteriori error estimator}
\label{subsection-apost-error-est}

Let a parameter ${\bmu} \in \Xi^{\train}$ be fixed and let us also fix the coarse node $z \in \mathcal{N}_H$. The solution space ${W}_h(U_k(K))$ is given according to (\ref{local-kernel-space}). By $\WKzRB=\mbox{\rm span}\{ Q_{h,k}^K(\Phi_z;{\bmu}_1), \ldots, Q_{h,k}^K(\Phi_z;{\bmu}_J) \}$ we denote an arbitrary reduced basis subspace of $W_h(U_k(K))$. An orthonormal basis of $\WKzRB$ shall be denoted by
$\{ \xi_1^K, \ldots, \xi_J^K \}$ (hence $\WKzRB = \mbox{\rm span}\{ \xi_1^K, \ldots, \xi_J^K \}$). Such an orthonormalization is required to avoid ill-conditioned stiffness matrices and hence numerical instabilities when solving a local reduced problem in $\WKzRB$. Given the original basis $\{ Q_{h,k}^K(\Phi_z;{\bmu}_1), \ldots, Q_{h,k}^K(\Phi_z;{\bmu}_J) \}$, an orthonormal basis $\{ \xi_1^K, \ldots, \xi_J^K \}$ can be obtained by a Gram-Schmidt process.

Now, we consider two multiscale scale basis functions associated with the node $z \in \mathcal{N}_H$, namely
$\Phi_z^{\ms}({\bmu}):=\Phi_z + Q_{h,k}(\Phi_z;{\bmu})$, where $Q_{h,k}(\Phi_z;{\bmu})$ is computed with the strategy proposed in Definition \ref{definition:localized:ms:space} and 
\begin{align*}
\Phi_z^{\ms,\rb}({\bmu}) = \Phi_z + Q^{\rb}_{h,k}(\Phi_z;{\bmu}), \quad \mbox{where} \quad
Q^{\rb}_{h,k}(\Phi_z;{\bmu}):=\underset{K \subset \omega_z}{\sum_{K\in \T_H}} Q^{K,\rb}_{h,k}(\Phi_z;{\bmu})
\end{align*}
and $Q^{K,\rb}_{h,k}(\Phi_z;{\bmu}) \in \WKzRB$ is the solution of the following problem:
\begin{align*}
\int_{U_k(K)} \aeps(\cdot \hspace{2pt}; {\bmu} ) \nabla Q^{K,\rb}_{h,k}(\Phi_z;{\bmu}) \cdot \nabla w = - \int_K \aeps(\cdot \hspace{2pt}; {\bmu} ) \nabla \Phi_z \cdot \nabla w \qquad \mbox{for all } w \in \WKzRB.
\end{align*}
We want to state an error estimator for the error between the \quotes{optimal} basis function $\Phi_z^{\ms}({\bmu})$ and its RB approximation $\Phi_z^{\ms,\rb}({\bmu})$. This can be obtained straightforwardly by using the the Riesz representative $r^K_z({\bmu}) \in W_h( U_k(K) )$ that is given as the solution of
\begin{align*}
(\nabla r^K_z({\bmu}), \nabla w_h )_{L^2(U_k(K))} = \int_{U_k(K)} \aeps(\cdot \hspace{2pt}; {\bmu} ) \nabla Q^{K,\rb}_{h,k}(\Phi_z;{\bmu}) \cdot \nabla w_h  + \int_K \aeps(\cdot \hspace{2pt}; {\bmu} ) \nabla \Phi_z \cdot \nabla w_h
\end{align*}
for all in $w_h \in W_h(U_k(K))$. Hence, it fulfills 
$$(\nabla r^K_z({\bmu}), \nabla (Q^{K,\rb}_{h,k}(\Phi_z;{\bmu}) - Q_{h,k}^{K}(\Phi_z;{\bmu})) )_{L^2(U_k(K))}^{1/2} = \| Q_{h,k}^{K}(\Phi_z;{\bmu}) - Q^{K,\rb}_{h,k}(\Phi_z;{\bmu}) \|_{\mathcal{E}(U_k(K))}^{{\bmu}}.$$
We next define 
\begin{align}
\label{de-omega-k-node}\omega^k_z:= \cup \{ U_k(K)| \hspace{2pt} K \in \T_H, \enspace K \subset \omega_z \},
\end{align}
and observe that
\begin{align}
\nonumber
\| \Phi_z^{\ms}({\bmu}) - \Phi_z^{\ms,\rb}({\bmu}) \|_{\mathcal{E}({\omega^k_z})}^{{\bmu}}
&= \left( \int_{{\omega^k_z}} \left| \underset{K \subset \omega_z}{\sum_{K \in \T_H}} \aeps(\cdot,{\bmu})^{1/2} \nabla (Q_{h,k}^{K}(\Phi_z;{\bmu}) - Q^{K,\rb}_{h,k}(\Phi_z;{\bmu})) \right|^2 \right)^{1/2}\\
\nonumber
&\le \sqrt{C_z} \underset{K \subset \omega_z}{\sum_{K \in \T_H}} 
 \| Q_{h,k}^{K}(\Phi_z;{\bmu}) - Q^{K,\rb}_{h,k}(\Phi_z;{\bmu}) \|_{\mathcal{E}(U_k(K))}^{{\bmu}}\\
 \label{equ:energy_residual}
 &\le \sqrt{\frac{C_z}{\alpha}}  \underset{K \subset \omega_z}{\sum_{K \in \T_H}} \| \nabla r^K_z({\bmu}) \|_{L^2(U_k(K))},
\end{align} 
{where $C_z$ is a constant that only depends on the number of elements in $\omega_z^k$}.
It remains to discuss an efficient computation of the Riesz representative $r^K_z({\bmu}) \in W_h(U_k(K))$, where we exploit the affine representation of $\aeps(\cdot\hspace{2pt};{\bmu})$ (see also \cite[Section 4.4]{PaR07}). Recall $\aeps(x;{\bmu}) = \sum_{q \in \mathcal{Q}} \Theta_q({\bmu}) \aeps_q(x)$. The idea is to compute a set of  \quotes{Riesz representatives basis} that can be reused for later computations. 

Let $Q^{K,\rb}_{h,k}(\Phi_z;{\bmu}) = \sum_{j=1}^J c_j({\bmu}) \xi_j^K$, where $\{ \xi_1^K, \ldots, \xi_J^K \}$ is the orthonormal basis of $\WKzRB$ that we introduced at the beginning of this subsection. First, we compute the representatives $l^K_{q,z} \in W_h(U_k(K))$ (for $q \in \mathcal{Q}$) by
\begin{align*}
(\nabla l^K_{q,z}, \nabla w_h )_{L^2(U_k(K))} = \int_K \aeps_q \nabla \Phi_z \cdot \nabla w_h \qquad \mbox{for all } w_h \in W_h( U_k(K) )
\end{align*}
and the representatives $h^K_{q,j,z} \in W_h(U_k(K))$ (for $q \in \mathcal{Q}$ and $1 \le j \le J$) by
\begin{align*}
(\nabla h^K_{q,j,z}, \nabla w_h )_{L^2(U_k(K))} = \int_{U_k(K)} \aeps_q \nabla \xi_j^K \cdot \nabla w_h \qquad \mbox{for all } w_h \in W_h(U_k(K)).
\end{align*}
The functions $l^K_{q,z}$ and $h^K_{q,j,z}$ are hence independent of the parameter ${\bmu}$. Consequently, the parameter-dependent Riesz representative $r^K_z({\bmu})$ can be expressed as
$$
r^K_z({\bmu}) = \sum_{q \in \mathcal{Q}}  \Theta_q({\bmu}) \left( l^K_{q,z} 
+ \sum_{j=1}^J c_j({\bmu}) h^K_{q,j,z} \right).
$$
Observe that, if $l^K_{q,z}$ and $h^K_{q,j,z}$ are precomputed, the residual $r^K_z({\bmu})$ can be directly evaluated by using the formula, without solving an additional system of equations. The precomputation happens purely in the offline phase and is hence very cheap in the online phase.

\begin{notation}\label{double-notation-onb}From now on, we slightly abuse the notation and denote by $\{ Q_{h,k}^K(\Phi_z;{\bmu}_1), \ldots,$ $Q_{h,k}^K(\Phi_z;{\bmu}_J) \}$ an orthonormal basis of $\WKzRB$. We make use of this simplification to avoid an additional notation in the subsequent sections. Since an orthonormal basis can be always straightforwardly obtained from $\{ Q_{h,k}^K(\Phi_z;{\bmu}_1), \ldots, Q_{h,k}^K(\Phi_z;{\bmu}_J) \}$ (even with a canonic numbering) we can make use of this double notation without loss of generality.\end{notation}

According to the previous discussion, the next step in the approach should be to compute the parameter-independent representatives $l^K_{q,z}$ and the representative $h^K_{q,1,z}$ for the initial parameter ${\bmu}_1$ that we selected in Step \ref{step-2}. Hence, we make the following step.

\begin{step}[Offline - Compute initial Riesz representatives]
\begin{algo}
 \rule{0.95\textwidth}{.7pt} \\
Algorithm: initialRieszRepresentatives( $\{ \WKzRB| \hspace{2pt} z \in \mathcal{N}_H; \hspace{2pt} K \in \T_H, \enspace K \subset \omega_z\}$ )
 \rule{0.95\textwidth}{.7pt} \\
\ForEach{$q \in \mathcal{Q}$}
{ 
In parallel \ForEach{$z \in \mathcal{N}_H$}
{
\ForEach{$K \in \T_H$ \mbox{\rm with} $K \subset \omega_z$}
{
Compute $l^K_{q,z} \in W_h(U_k(K))$ by
\begin{align*}
(\nabla l^K_{q,z}, \nabla w_h )_{L^2(U_k(K))} = \int_K \aeps_q \nabla \Phi_z \cdot \nabla w_h \qquad \forall w_h \in W_h(U_k(K)).
\end{align*}
and $h^K_{q,1,z} \in W_h(U_k(K))$ by
\begin{align*}
(\nabla h^K_{q,1,z}, \nabla w_h )_{L^2(U_k(K))} = \int_{U_k(K)} \aeps_q \nabla Q_{h,k}^{K}(\Phi_z;{\bmu}_1) \cdot \nabla w_h \qquad \forall w_h \in W_h(U_k(K)).
\end{align*}
}}}
 \rule{0.95\textwidth}{.7pt}
\end{algo}
\end{step}

Based on precomputed representatives $l^K_{q,z}$ and $h^K_{q,j,z}$, we now introduce a corresponding local error indicator $\triangle_{z,{\bmu}}$ according to the findings from Section \ref{subsection-apost-error-est}.

\begin{definition}\label{definition-residual-error-indicator}
For $z \in \mathcal{N}_H$ and $K\in \mathcal{T}_H$ with $K \subset \omega_z$, let $\WKzRB$ be a corresponding local Reduced Basis space with basis $\{ Q_{h,k}^{K}(\Phi_z;{\bmu}_{z,j}) | \hspace{2pt} 1 \le j \le J\}$. For a given element $w({\bmu}) \in \WKzRB$ that is represented by
\begin{align*}
w({\bmu}) =  \sum_{j=1}^J c_j({\bmu}) \hspace{2pt} Q_{h,k}^{K}(\Phi_z;{\bmu}_{z,j}),
\end{align*}
we define the corresponding residual error indicator $\triangle_{z,{\bmu}}$ by
$$\triangle_{z,{\bmu}} :={\frac{\sqrt{C_z}}{\alpha}}  \underset{K \subset \omega_z}{\sum_{K \in \T_H}} \| \nabla r^K_z({\bmu}) \|_{L^2(U_k(K))},\quad 
\mbox{where}
\quad
r^K_z({\bmu}) = \sum_{q \in \mathcal{Q}}  \Theta_q({\bmu}) \left( l^K_{q,z} 
+ \sum_{j=1}^J c_j({\bmu}) h^K_{q,j,z} \right).
$$
\end{definition}

The explicit computation of $l^K_{q,z}$ and $h^K_{q,j,z}$ is described in Step \ref{step-4}.
{From equation \eqref{equ:energy_residual} together with an easy computation using Assumption (A2) we have the following upper and lower bounds for the residual error indicator that are crucial to show an apriori estimate for the Greedy algorithm
\begin{equation}
\label{equ:upper_lower}
\|\nabla \Phi_z^{\ms}({\bmu}) - \nabla\Phi_z^{\ms,\rb}({\bmu}) \|_{L^2({\omega^k_z})}\leq \triangle_{z,{\bmu}}
\leq \sqrt{C_z}\frac{\beta}{\alpha}\|\nabla \Phi_z^{\ms}({\bmu}) - \nabla\Phi_z^{\ms,\rb}({\bmu}) \|_{L^2({\omega^k_z})}.
\end{equation}
}

\subsection{Greedy search algorithm}

The next step describes a classical Greedy search procedure formulated for our setting. The idea is to start from a given parameter set $ \Xi_{z}^{\rb}$ and a corresponding local RB space $\WKzRB$. Then we solve the local problem in the RB space for {\it every} parameter in $\Xi^{\train} \setminus \Xi_{z}^{\rb}$ and use the error estimator $\triangle_{z,{\bmu}}$ to find out for which of these parameters we make the biggest error. This parameter is relevant and should be hence added to $\Xi_{z}^{\rb}$. The corresponding exact solution of the local problem (i.e the solution in $W_h(U_k(K))$) is consequently added to $\WKzRB$.
{We hence need to solve three different types of problems in Step \ref{step-4} below. The first type involves the local RB spaces $\WKzRB$ (for given node $z \in \mathcal{N}_H$ and element $K \in \T_H$ with $K \subset \omega_z$). Here, we solve for the analog of problem (\ref{local-corrector-problem}) in $\WKzRB$, i.e. we compute $Q^{K,\rb}_{h,k}(\Phi_z;{\bmu}) \in \WKzRB$ with
\begin{align}
\label{local-corrector-problem-in-rb-space}\int_{U_k(K)} \aeps(\cdot \hspace{2pt}; {\bmu} ) \nabla Q^{K,\rb}_{h,k}(\Phi_z;{\bmu}) \cdot \nabla w = - \int_K \aeps(\cdot \hspace{2pt}; {\bmu} ) \nabla \Phi_z \cdot \nabla w
\end{align}
for all $w \in \WKzRB$. Second, we need to solve for the contributions $h^K_{q,j,z}$ of the Riesz representatives, i.e. for $q\in \mathcal{Q}$, $z\in \mathcal{N}_H$, $K\in \T_H$ and a given parameter ${\bmu}_{z,j}$, find $h^K_{q,j,z} \in W_h(U_k(K))$ with
\begin{align}
\label{problem-in-greedy-h-qjz}(\nabla h^K_{q,j,z}, \nabla w_h )_{L^2(U_k(K))} = \int_{U_k(K)} \aeps_q \nabla Q_{h,k}^{K}(\Phi_z;{\bmu}_{z,j}) \cdot \nabla w_h.
\end{align}
for all $w_h \in W_h(U_k(K))$. The third type of problems involves the standard local problems (\ref{local-corrector-problem}), which need to be solved for new parameters that are added to the parameter set.}

\begin{remark}
Step \ref{step-4} below involves to solve problems in the local RB spaces $\WKzRB$. This requires the assembly of corresponding dense stiffness matrices with entries 
$$\int_{U_k(K)} \aeps(\cdot \hspace{2pt}; {\bmu} ) \nabla \xi_z({\bmu}_j) \cdot \nabla \xi_z({\bmu}_i),$$
where 
%
%
$\xi_z({\bmu}_i)=Q_{h,k}^{K}(\Phi_z;{\bmu}_{i})$
with ${\bmu}_i \in  \Xi_{z}^{\rb}$ denote the elements of an 
orthonormal 
basis of $\WKzRB$.
Here we can use the affine representation of $\aeps$ to make this procedure more efficient. After each step of the iterative procedure to update $\WKzRB$ (i.e. the step $J \mapsto J+1$), we just have to assemble and store the new entries
$$
\int_{U_k(K)} \aeps_q \nabla \xi_z({\bmu}_{J+1}) \cdot \nabla \xi_z({\bmu}_i) \qquad \mbox{for } 1 \le i \le J \enspace \mbox{and} \enspace
q \in \mathcal{Q}. 
$$
Like that, the required system matrices (for the problems in $\WKzRB$) are straightforwardly (and cheaply) obtained by summation over $q$.
\end{remark}

An analysis of the Greedy procedure is presented in \cite{BMP12}. Rates can be measured by the Kolmogorov $n$-width that describes how good a subset $F$ of a Hilbert space $X$ can be approximated by an $n$-dimensional subspace $Y_n$. In our case we have 
$$W_z := \{ Q_{h,k}(\Phi_z;\boldsymbol{\mu})| \hspace{2pt} \boldsymbol{\mu} \in \Xi^{\train} \}$$
and the corresponding Kolmogorov $n$-width in $W_h(\omega^k_z)$ is defined by
$$
d_n( W_z, W_h(\omega^k_z) ) := \inf \{ \sup_{w \in F_z} \inf_{v \in Y_{z,n}} \| v - w\|_{H^1(\omega^k_z)} | \hspace{2pt} Y_{z,n} \mbox{ is $n$-dim subspace of } W_h(\omega^k_z) \}.
$$
If we define a space $W_z^n$ as the span of the elements 
obtained by the Greedy procedure
 i.e.
$$W_{z}^n := \mbox{\rm span} \{ Q_{h,k}(\Phi_z;{\bmu}_{z,1}), \ldots, Q_{h,k}(\Phi_z;{\bmu}_{z,n}) \};$$
then {following} \cite{BMP12} it  holds  
{$$\sup_{{\psi\in {W}_z} } \inf_{w \in W_{z}^n} \| \nabla {\psi}  - \nabla w \|_{L^2(\omega^k_z)} \leq {C^{n+1}_{\alpha,\beta,C_z}}\hspace{2pt}(n+1) \hspace{2pt}\hspace{2pt}d_n( {W}_z, W_h(\omega^k_z) ).$$}
{Precisely using \eqref{equ:upper_lower} (following the lines of the proof in \cite{BMP12}) one obtains
\begin{equation}
\label{equ: cst_Kolm}
C^{n+1}_{\alpha,\beta,C_z}=\left(1+\sqrt{C_z}\frac{\beta}{\alpha}\sqrt{\frac{\beta}{\alpha}}\right)^{n+1}.
\end{equation}
}

{Thus we see from \eqref{equ: cst_Kolm} that
that the error inherited from the Kolmogorov $n$-width can be polluted by a factor of order 
$C^{n+1}_{\alpha,\beta,C_z} (n+1)$.} However, it is not possible to give a general answer about the size of $d_n( F_z, W_h(\omega^k_z) )$ itself. For smooth dependencies of $\aeps$ on the parameter ${\bmu}$, typically exponential convergence rates can be numerically observed, {leading also to an exponential convergence of the error. The usually justified assumption can be hence stated as follows: there exists some $\gamma_z>0$ and a source-term depending constant $C(\Phi_z)>0$ such that
$$d_n( F_z, W_h(\omega^k_z) ) \lesssim C(\Phi_z) e^{-\gamma_z n}.$$
Note that we can bound $C(\Phi_z)\lesssim \| \nabla \Phi_z \|_{L^2(\omega_z)}$ because $\| \nabla Q_{h,k}(\Phi_z;{\bmu}_{z,j}) \|_{L^2(\omega_z^k)} \lesssim \| \nabla \Phi_z \|_{L^2(\omega_z)}$ for all parameters ${\bmu}_{z,j}$. Combining these results, we can apply \cite[Theorem 3.1 and Corollary 4.1]{BMP12} to state the following proposition. 
\begin{proposition}\label{proposition-kolmogorov-exp-convergence}
Let $z \in \mathcal{N}_H$ and 
{$W_z:=\{ Q_{h,k}(\Phi_z;{\bmu})| \hspace{2pt} \bmu \in \Xi^{\train} \}$.}
The $J$-dimensional local RB space
obtained with the Greedy strategy is denoted by
$\WzRB:= \mbox{\rm span}\{ Q_{h,k}(\Phi_z;{\bmu}_{z,j})| \hspace{2pt} 1 \le j \le J \}$.
If the Kolmogorov $n$-width of $F_z$ in $W_h(\omega_z^k)$ satisfies
$$d_J( F_z, W_h(\omega_z^k) ) \lesssim \| \nabla \Phi_z \|_{L^2(\omega_z)} e^{-\gamma_z J}$$
{with $\gamma_z>\log (1+\sqrt{C_z}\frac{\beta}{\alpha}\sqrt{\frac{\beta}{\alpha}})$,
then  there exists $\zeta_z>0$} such that for all ${\bmu} \in \Xi^{\train}$ it holds
\begin{align}
\label{exp-convergence-in-loc-rb-space}
\inf_{w \in \WzRB} \| \nabla Q_{h,k}(\Phi_z;{\bmu})  - \nabla w \|_{L^2(\omega_z^k)} \lesssim \| \nabla \Phi_z \|_{L^2(\omega_z)} e^{-\zeta_z J}.
\end{align}
\end{proposition}}

\begin{remark}[Proper Orthogonal Decomposition (POD)]
As an alternative to the above described Greedy procedure, it is also possible to use the Proper Orthogonal Decomposition, also known as Karhunen-Lo\`{e}ve decomposition, to identify suitable parameter sets $\Xi_{z}^{\rb}$. The POD (cf. \cite{Jol86,Rav00,Sir87}) is based on the following question: given a space $V$, what is the optimal subspace of dimension $n\in \mathbb{N}$ so that the error of an orthogonal projection onto this space is minimized? Practically, this leads to a number of eigenvalue problems that need to be solved. The eigenvectors to the $n$ first eigenvalues span the desired subspace. An application of this strategy to the Multiscale Finite Element Method (MsFEM) can be found in \cite{Ngu08}.
\end{remark}

Observe that thanks to the exponential convergence of the Kolmogorov $n$-width (for smooth dependencies of $\aeps$ on ${\bmu}$), the dimensionality of the parameter set $\mathcal{D}$ is typically not a big problem. Due to the exponential convergence, the RB-parameter sets $\Xi^{\rb}_z$ can be expected to remain small. Hence, solving a local problem for an arbitrary parameter in the training set (that is not in $\Xi^{\rb}_z$) only involves solving a problem of very small dimension. These costs are negligible as long as the training set $\Xi^{\train}$ is of moderate size.

{\begin{step}[Offline - Greedy search]\label{step-4}
\begin{algo}
 \rule{0.9\textwidth}{.7pt} \\
Set ${\bmu}_{z,1}:={\bmu}_{1}$ for all $z \in \mathcal{N}_H$.\\
Recall $\WKzRB=\mbox{\rm span}\{ Q_{h,k}^K(\Phi_z;{\bmu}_{z,1}) \}$ 
for all $z \in \mathcal{N}_H$ and all $K \in \T_H$ with $K \subset \omega_z$.
 \rule{0.9\textwidth}{.7pt} \\
Algorithm: greedyLoop( $\{ \Xi_{z}^{\rb}| \hspace{2pt} z \in \mathcal{N}_H\}$, TOL )
 \rule{0.9\textwidth}{.7pt} \\
In parallel \ForEach{$z \in \mathcal{N}_H$}
{
\While{ $\underset{{{\bmu} \in \Xi^{\train}} }{\max} \triangle_{z,{\bmu}} > \mbox{\rm TOL}$ }{
Set $J:=|\Xi_{z}^{\rb}|$.\\
\ForEach{$K \in \T_H$ \mbox{\rm with} $K \subset \omega_z$}
 {
\ForEach{${\bmu} \in \Xi^{\train} \setminus \Xi_{z}^{\rb}$}
 {
compute $Q^{K,\rb}_{h,k}(\Phi_z;{\bmu}) \in \WKzRB$ via (\ref{local-corrector-problem-in-rb-space}).
   }
 }
\ForEach{${\bmu} \in \Xi^{\train} \setminus \Xi_{z}^{\rb}$}
{
 Compute $\triangle_{z,{\bmu}}$ via Definition \ref{definition-residual-error-indicator}.
 }
 Set ${\bmu}_{z,J+1} = \underset{{\bmu} \in \Xi^{\train} \setminus \Xi_{z}^{\rb}}{\mbox{\rm argmax}} \triangle_{z,{\bmu}}$ and set $\Xi_{z}^{\rb}:=\Xi_{z}^{\rb} \cup \{ {\bmu}_{z,J+1} \}$.\\
Update Riesz representatives.\\
\ForEach{$K \in \T_H$ \mbox{\rm with} $K \subset \omega_z$}
 {
   Compute $Q_{h,k}^{K}(\Phi_z;{\bmu}_{z,J+1}) \in {W}_h(U_k(K))$ \mbox{\rm via (\ref{local-corrector-problem}).}\\
   Set 
   $\WKzRB:= \WKzRB \oplus \mbox{\rm span}\{ Q_{h,k}^K(\Phi_z;{\bmu}_{z,J+1}) \}.$\\
   For all $q\in \mathcal{Q}$, compute $h^K_{q,J+1,z} \in W_h(U_k(K))$ via (\ref{problem-in-greedy-h-qjz}).
 }
 Set $\WzRB:=\mbox{\rm span}\{ Q_{h,k}(\Phi_z;{\bmu})| \hspace{2pt} {\bmu} \in \Xi_{z}^{\rb} \}$, where
 $$Q_{h,k}(\Phi_z;{\bmu}):=\underset{K \subset \omega_z}{\sum_{K\in \T_H}} Q_{h,k}^{K}(\Phi_z;{\bmu}).$$
}
}
 \rule{0.9\textwidth}{.7pt}
\end{algo}
\end{step}}

\begin{remark}
One might pose the question, if the the local parameter sets $\Xi^{\rb}_z$ can be chosen identically for each of the coarse nodes $z\in \mathcal{N}_H$. This is possible, but typically not efficient for the online phase. Since $\aeps(x;\bmu)$ has a possibly heterogenous dependency on $x$, a certain parameter might be crucial in an environment of a certain coarse note, but might trigger no changes in the environment of another coarse node. Hence, using a shared parameter set for all nodes, will probably lead to local RB-spaces that are larger than necessary and hence increases the computational complexity. 

However, the situation can be different if $\aeps(x;\bmu)$ yields some structural assumptions
in the $x$-dependency (such as quasi-periodicity or ergodicity). In this case we might exploit these features to assemble only one (shared) RB parameter set, based on representative computations in only one (or few) of the patches. Furthermore, in such a setting it might be even possible to map the local solutions in one patch to the solutions in another patch by a simple transformation (without recomputing it). This can indeed lead to an enormous reduction of the computational complexity. 
\end{remark}

\subsection{Precomputation of local stiffness matrix and load vector entries}
\label{subsection-precomputation-local-entries}

Recall Notation \ref{double-notation-onb}.
After Step \ref{step-4} the local RB spaces $\WzRB$ are assembled for every coarse node $z \in \mathcal{N}_H$. In the online phase, for a given parameter ${\bmu} \in \mathcal{D}$, local problems are solved in $\WzRB$.
 To make this computation efficient in the online phase, it is necessary to pre-assemble the values of the corresponding stiffness matrix and the associated right hand sides. Again, we can use the affine representation of $\aeps$. The procedure is summarized in Step \ref{step-5}.

\begin{step}[Offline - Preassembly of local stiffness matrix and load vector entries]\label{step-5}
\begin{algo}
\rule{0.9\textwidth}{.7pt} \\
We assemble the matrices $D^{z,q} \in \R^{J_z \times J_z}$, where $J_z= |\Xi_{z}^{\rb}|$.\\
We assemble the load vectors $F^{z,q} \in \R^{J_z}$, where $J_z= |\Xi_{z}^{\rb}|$.
\rule{0.9\textwidth}{.7pt} \\
 Algorithm: precomputationLocalSystemMatrices( $\{ \WzRB| \hspace{2pt} z \in \mathcal{N}_H\}$ )
\rule{0.9\textwidth}{.7pt} \\
In parallel \ForEach{$z \in \mathcal{N}_H$}
{
\For{$q \in \mathcal{Q}$}
{
\For{$j = 1, \ldots, |\Xi_{z}^{\rb}|$}
{
Assemble $F_{j}^{z,q}:=\int_{\Omega} \aeps_q \nabla \Phi_{z} \cdot \nabla Q_{h,k}(\Phi_{z};{\bmu}_{z,j})$.\\
\For{$i = j, \ldots, |\Xi_{z}^{\rb}|$}
{
Assemble $D_{ji}^{z,q}=D_{ij}^{z,q}:=\int_{\Omega} \aeps_q \nabla Q_{h,k}(\Phi_{z};{\bmu}_{z,j}) \cdot \nabla Q_{h,k}(\Phi_{z};{\bmu}_{z,i})$.
}}
}}
\rule{0.9\textwidth}{.7pt}
\end{algo}
\end{step}
Assume that we want to solve a problem in the space $\WzRB$, which is of the structure: 
\begin{align*}
\mbox{find } v^{\rb} \in \WzRB: \qquad \beps( v^{\rb}, w; \bmu ) = - \beps( \Phi_z, w; \bmu ) \qquad \forall w \in \WzRB,
\end{align*}
for some coarse nodal basis function $\Phi_z$. Then, Step \ref{step-5} allows to write down the corresponding stiffness matrix $S_z(\bmu)$ by pure summation: $S_z(\bmu) = \sum_{q \in \mathcal{Q}} \Theta_q(\bmu) D^{z,q}$. No quadrature rule is required anymore. The assembly of $S_z(\bmu)$ is thus fast in the online phase. The same holds for the load vector.  Since the dimension of $\WzRB$ is typically small, the inversion of $S_z(\bmu)$ is cheap e.g., LU type decomposition can often be used.

\subsection{Precomputation of global stiffness matrix entries}

In the previous two steps, we computed the local Reduced Basis spaces $\WzRB$ for $z\in \mathcal{N}_H$ and basically pre-assembled the corresponding stiffness matrices and load vectors. In the online phase, for a given parameter ${\bmu} \in \mathcal{D}$, we want to construct
$$\VmsRB(\bmu)=\mbox{\rm span}\{ \Phi_{z}^{\rb}({\bmu}) \hspace{2pt} z \in \mathcal{N}_H\}.$$
where  the RB multiscale basis function is given by the following problem: find $\Phi_{z}^{\rb}({\bmu}) \in \Phi_z + \WzRB$ such that 
\begin{align}
\label{local-rb-multiscale-problem} \beps(\Phi_{z}^{\rb}({\bmu}), w ; {\bmu} ) = 0 \qquad \forall w \in \WzRB.
\end{align}
Various terms of the entries of the stiffness matrix $\beps(\Phi_{z_m}^{\rb}({\bmu}), \Phi_{z_n}^{\rb}({\bmu}) , {\bmu}), \enspace \mbox{for } 1\le n,m \le N,$ can be pre-computed in the offline stage as explained below.
In this section, we hence discuss the pre-computation of certain terms that are required in the online phase. We again recall Notation \ref{double-notation-onb}.

For this purpose let $\Phi_{z_m}^{\rb}({\bmu})$ be given by
$$
\Phi_{z_m}^{\rb}({\bmu}) = \Phi_z + \sum_{{\bmu}_{z,j} \in \Xi_{z}^{\rb}} c_{z}({\bmu}) Q_{h,k}(\Phi_z;{\bmu}_{z,j}).
$$
With this, we can write the entries of the system matrix by
\begin{eqnarray}
\nonumber\lefteqn{\beps(\Phi_{z_m}^{\rb}({\bmu}), \Phi_{z_n}^{\rb}({\bmu}) , {\bmu})
= \sum_{q \in \mathcal{Q}} \Theta_q({\bmu}) \int_{\Omega} \aeps_q \nabla \Phi_{z_m} \cdot \nabla \Phi_{z_n}}\\
\label{entries-global-stiffness-matrix-1}&+& \sum_{q \in \mathcal{Q}} \sum_{{\bmu}_{z_m,j} \in \Xi_{z_m}^{\rb}}  \Theta_q({\bmu}) c_{z_m,j}({\bmu}) \int_{\Omega} \aeps_q \nabla Q_{h,k}(\Phi_{z_m};{\bmu}_{z_m,j})  \cdot \nabla \Phi_{z_n} \\
\nonumber&+& \sum_{q \in \mathcal{Q}} \sum_{{\bmu}_{z_n,i} \in \Xi_{z_n}^{\rb}} \Theta_q({\bmu}) c_{z_n,i}({\bmu})  \int_{\Omega} \aeps_q \nabla \Phi_{z_m} \cdot \nabla Q_{h,k}(\Phi_{z_n};{\bmu}_{z_n,i}) \\
\nonumber&+& \sum_{q \in \mathcal{Q}} \sum_{{\bmu}_{z_m,j} \in \Xi_{z_m}^{\rb}} \sum_{{\bmu}_{z_n,i} \in \Xi_{z_n}^{\rb}} \Theta_q({\bmu}) c_{z_m,j}({\bmu}) c_{z_n,i}({\bmu})  \int_{\Omega} \aeps_q \nabla Q_{h,k}(\Phi_{z_m};{\bmu}_{z_m,j})  \cdot \nabla Q_{h,k}(\Phi_{z_n};{\bmu}_{z_n,i}). 
\end{eqnarray}
Hence, independent of ${\bmu} \in \mathcal{D}$, we can precompute 
\begin{align}
\label{entries-global-stiffness-matrix-2}\nonumber S_{nm}^q=S_{mn}^q&:=\int_{\Omega} \aeps_q \nabla \Phi_{z_m} \cdot \nabla \Phi_{z_n};\\
M_{nm}^q(i,j) = M_{mn}^q(j,i) &:= \int_{\Omega} \aeps_q \nabla Q_{h,k}(\Phi_{z_m};{\bmu}_{z_m,j})  \cdot \nabla Q_{h,k}(\Phi_{z_n};{\bmu}_{z_n,i})\\
\nonumber\mbox{and} \hspace{40pt}
R_{nm}^q(j) &:= \int_{\Omega} \aeps_q \nabla Q_{h,k}(\Phi_{z_m};{\bmu}_{z_m,j})  \cdot \nabla \Phi_{z_n}
\end{align}
for $1 \le n,m \le N$ and for $1\le i \le |\Xi_{z_n}^{\rb}|$ and $1\le j \le |\Xi_{z_m}^{\rb}|$. Once these entries are precomputed, the global stiffness matrix entry $\beps(\Phi_{z_m}^{\rb,\ms}({\bmu}), \Phi_{z_n}^{\rb,\ms}({\bmu}) , {\bmu})$ can be again obtained via a simple summation. Note that we did not assume an affine representation of the source $f(\cdot\hspace{2pt};\bmu)$. Hence, we can typically not precomput it. However, this is no problem since the the reduced basis multiscale basis functions $\Phi_{z}^{\rb}({\bmu})$ are only locally supported. Hence, the usage of a quadrature rule for assembling the load vector of the global problem is not very costly and can be performed on-the-fly in the online phase. Furthermore, if an affine representation of $f(\cdot\hspace{2pt};\bmu)$ is available, it can be exploited and the entries of the load vector can be precomputed in the same way as for the stiffness matrix.

Using equations (\ref{entries-global-stiffness-matrix-1}) and (\ref{entries-global-stiffness-matrix-2}) it is straightforwardly possible to compute the global stiffness matrix by simple summation. We have
\begin{eqnarray}
\label{global-stiffness-matrix-entry}\lefteqn{\beps(\Phi_{z_m}^{\rb}({\bmu}), \Phi_{z_n}^{\rb}({\bmu}) , {\bmu})}\\
\nonumber&=& \sum_{q \in \mathcal{Q}} \Theta_q({\bmu}) 
\left( S_{nm}^q + \sum_{{\bmu}_{z_m,j} \in \Xi_{z_m}^{\rb}}  c_{z_m,j}({\bmu}) R_{nm}^q(j)
+ \sum_{{\bmu}_{z_n,i} \in \Xi_{z_n}^{\rb}} c_{z_n,i}({\bmu}) R_{mn}^q(i)\right.\\
\nonumber&\enspace& \quad \left. + \sum_{{\bmu}_{z_m,j} \in \Xi_{z_m}^{\rb}} \sum_{{\bmu}_{z_n,i} \in \Xi_{z_n}^{\rb}} c_{z_m,j}({\bmu}) c_{z_n,i}({\bmu}) M_{nm}^q(i,j) \right).
\end{eqnarray}

\subsection{The online phase}
\label{subsection-online-phase}

Let $\bmu \in \mathcal{D}$ be arbitrary given parameter for which we want to obtain a multiscale approximation of the original problem (\ref{equation-weak}). {This can be achieved in an efficient way using the pre-assembled quantities from Step \ref{step-1}-\ref{step-5} and equation \eqref{entries-global-stiffness-matrix-2} (i.e. the localized RB spaces $\WzRB$ and the corresponding pre-computed system matrix and load vector entries).}

{\it Local problems.} We recall that we have to compute the basis functions $\Phi_{z}^{\rb}({\bmu}) \in \Phi_z + \WzRB$ for $z=1,\ldots,N$ solutions of problem \eqref{local-rb-multiscale-problem}.
This can be done with low costs using the results from equation \eqref{entries-global-stiffness-matrix-1} and \eqref{entries-global-stiffness-matrix-2}. Recall (for $z \in \mathcal{N}_H$ and $q \in \mathcal{Q}$) the precomputed matrices $D^{z,q}$ and load vectors $F^{z,q}$ (see Step \ref{step-5}). We define $D^{z}(\bmu):=\sum_{q \in \mathcal{Q}} \Theta_q(\bmu) D^{z,q}$ and $F^{z}(\bmu):=\sum_{q \in \mathcal{Q}} \Theta_q(\bmu) F^{z,q}$. With that, solving for the solution $\Phi_{z}^{\rb,\ms}({\bmu})$ of (\ref{local-rb-multiscale-problem}) is equivalent to solving for $q_z(\bmu) \in \R^{J_z}$ with
\begin{align}
\label{rb-precomputed-local-problem}D^{z}(\bmu) q_z(\bmu) = F^{z}(\bmu) \qquad \mbox{and defining } 
\quad \Phi_z^{\rb}(\mu) := \Phi_z + \sum_{j=1}^{J_z} \hspace{2pt} (q_z(\bmu))_j \hspace{2pt} Q_{h,k}(\Phi_{z};{\bmu}_{z,j}).
\end{align}
The matrix $D^{z}(\bmu)$ is low dimensional and cheap to invert. No saddle point solver is required. The accuracy of $\Phi_z^{\rb}(\mu)$ can be checked using the error estimator $\triangle_{z,{\bmu}}$ defined in Definition \ref{definition-residual-error-indicator}.\\

\noindent{\it Global problem.} Once the RB basis functions $\Phi_z^{\rb}(\mu)$ are computed we can define
\begin{align}
\label{localized-ms-space-rb}
\VmsRB(\bmu):=\mbox{\rm span}\{ \Phi_{z}^{\rb}({\bmu})| \hspace{2pt} z \in \mathcal{N}_H\}
\end{align}
and solve for $\uRB(\cdot\hspace{2pt};\bmu) \in \VmsRB(\bmu)$ with
\begin{align}
\label{global-rb-multiscale-problem} \beps(\uRB(\cdot\hspace{2pt};\bmu), v ; {\bmu} ) = \left( f(\cdot\hspace{2pt};\bmu), v \right)_{L^2(\Omega)} \qquad \forall v \in \VmsRB(\bmu).
\end{align}
The solving can be done in an efficient way using the precomputed terms from \eqref{entries-global-stiffness-matrix-2}. We define the entries of the global stiffness matrix $S(\bmu) \in \R^{N \times N}$ by
\begin{align}
\label{global-rb-stiffness-matrix}S(\bmu)_{nm} := \beps(\Phi_{z_m}^{\rb}({\bmu}), \Phi_{z_n}^{\rb}({\bmu}) , {\bmu}), \qquad \mbox{exploiting formula (\ref{global-stiffness-matrix-entry}).}
\end{align}
The entries of the global load vector $F(\bmu) \in \R^{N}$ is given by
\begin{align}
\label{global-rb-load-vector}F(\bmu)_{n} :=  \left( f(\cdot\hspace{2pt};\bmu), \Phi_{z_n}^{\rb}({\bmu}) \right)_{L^2(\Omega)}.
\end{align}
With that we solve for ${{\boldsymbol u}}^{\rb} \in \R^{N}$
\begin{align}
\label{global-rb-multiscale-problem-algebraic} S(\bmu) \hspace{2pt} {{\boldsymbol u}}^{\rb} = F(\bmu)
\end{align}
and can set $\uRB(\cdot\hspace{2pt};\bmu):= \sum_{n=1}^N {{\boldsymbol u}}^{\rb}_n \hspace{2pt} \Phi_{z_n}^{\rb}({\bmu})$.

\begin{remark}
Even though the method is specifically designed to solve parameter-dependent multiscale problems, it can be also used to solve nonlinear and time-dependent multiscale problems. An example of how the method can be used to treat a nonlinear equation is presented in Section \ref{ss-mod-prob-2}. The treatment of time-dependent problems is similar. After introducing a suitable time discretization (e.g. by an implicit Euler scheme), the problem that needs to be solved in every time step is typically a standard elliptic problem. If we now interpret the time-dependency as an additional parameter, we obtain a parametrized elliptic multiscale problem for each time step. Hence, a time-dependent (and possibly additionally parameter dependent) equation can be seen as a large set of parametrized stationary problems. The application of the RB-LOD is obvious and each of the time steps becomes cheap.
\end{remark}

\begin{step}[Online phase]\label{online-step}
\begin{algo}
 \rule{0.95\textwidth}{.7pt} \\
 Compute a RB multiscale approximation for a given (online) parameter ${\bmu} \in \mathcal{D}$.
 \rule{0.95\textwidth}{.7pt} \\
Algorithm: getMultiscaleApproximation( $\{ \WzRB| \hspace{2pt} z \in \mathcal{N}_H\}$ )
 \rule{0.95\textwidth}{.7pt} \\
In parallel \ForEach{$z \in \mathcal{N}_H$}
{
Compute $\Phi_{z}^{\rb}({\bmu}) \in \Phi_z + \WzRB$ with
\begin{align*}
\beps( \Phi_{z}^{\rb}({\bmu}), w ; {\bmu} ) = 0 \qquad \forall w \in \WzRB
\end{align*}
using the precomputed formulation (\ref{rb-precomputed-local-problem}).
}
Set $\VmsRB(\bmu):=\mbox{\rm span}\{  \Phi_{z}^{\rb}({\bmu})| \hspace{2pt} z \in \mathcal{N}_H\}$.\\
Solve for the final RB multiscale approximation $\uRB(\cdot\hspace{2pt};\bmu) \in \VmsRB(\bmu)$ with
\begin{align*}
\beps(\uRB(\cdot\hspace{2pt};\bmu), v ; {\bmu} ) = \left( f(\cdot\hspace{2pt};\bmu), v \right)_{L^2(\Omega)} \qquad \forall v \in \VmsRB(\bmu).
\end{align*}
using the formulas (\ref{global-rb-stiffness-matrix})-(\ref{global-rb-multiscale-problem-algebraic}).
 \rule{0.95\textwidth}{.7pt}
\end{algo}
\end{step}

\subsection{A priori error analysis}
The convergence analysis for a multiscale reduced basis method usually combines an existing a priori error analysis for the parameter independent multiscale approximation with error estimates for the Greedy procedure such as stated in Proposition \ref{proposition-kolmogorov-exp-convergence} (see \cite{AbB12,AbB14,ABV14}).
The following result guarantees convergence of the method, independent of the variations in the coefficient $\aeps$.
\begin{theorem}
Assume (A1)-(A3), let $\bmu \in \Xi^{\train}$ and let the assumptions of Proposition \ref{proposition-kolmogorov-exp-convergence} be fulfilled. Furthermore we assume that there exists $J \in \mathbb{N}$ such that $J \lesssim \mbox{\rm dim} \WzRB \lesssim J$ for all nodes $z \in \mathcal{N}_H$. By $u_h(\cdot\hspace{2pt};\bmu) \in V_h$ we denote the reference solution, i.e. the solution of (\ref{reference-solution}), and by $\uRB(\cdot\hspace{2pt};\bmu) \in \VmsRB(\bmu)$ the RB-LOD approximation obtained in Step \ref{online-step}. If $k \ge C_{g} |\log(H)|$ as in Proposition \ref{prop-conv-lod} then it holds
\begin{align*}
\| \uRB(\cdot\hspace{2pt};\bmu) - u_h(\cdot\hspace{2pt};{\bmu}) \|_{H^1(\Omega)} \lesssim H + H^{-1} e^{-J \zeta} k^{d/2},
\end{align*}
where we define $0<\zeta := \min_{z \in \mathcal{N}_H} \zeta_z$ with $\zeta_z$ being the constant from Proposition \ref{proposition-kolmogorov-exp-convergence}. Obviously, if $J \ge 2 \zeta^{-1} |\log(H)|$ we preserve the linear convergence rate in $H$.
\end{theorem}
\begin{proof}
Consider the classical LOD basis function $\Phi_{z}^{\ms}({\bmu}):=\Phi_z + Q_{h,k}(\Phi_z;{\bmu})$ and its RB-LOD version $\Phi_{z}^{\rb}({\bmu}) \in \Phi_z + \WzRB$ as in Step \ref{online-step}. We define
$
\lambda_z^{\rb}({\bmu}) := \Phi_{z}^{\rb}({\bmu}) - \Phi_z \in \WzRB
$
and obtain with (\ref{exp-convergence-in-loc-rb-space})
\begin{align}
\label{main-theorem-step-1}\| \Phi_{z}^{\ms}({\bmu}) - \Phi_{z}^{\rb}({\bmu}) \|_{H^1(\omega_z^k)} = \| Q_{h,k}(\Phi_z;{\bmu}) - \lambda_z^{\rb}({\bmu}) \|_{H^1(\omega_z^k)} \lesssim  \| \nabla \Phi_z \|_{L^2(\omega_z)} e^{-\zeta_z J}.
\end{align}
Writing the classical LOD solution $u^{\ms}(\cdot\hspace{2pt};\bmu) \in \Vmsk({\bmu})$ (see problem \eqref{classical-LOD-solution-eq}) as
$$
u^{\ms}(\cdot\hspace{2pt};\bmu) = \sum_{z \in \mathcal{N_H}} u^{\ms}_z(\bmu) \Phi_{z}^{\ms}({\bmu}),
$$
and recalling that $\WzRB$ is a subspace of the kernel of the $L^2$-projection $P_{L^2}$ (cf. Remark \ref{remark-on-L2-projection}) we have
\begin{align}
\nonumber\| \sum_{z \in \mathcal{N_H}} u^{\ms}_z(\bmu) \Phi_{z}({\bmu}) \|_{L^2(\Omega)}
&= \| \sum_{z \in \mathcal{N_H}} u^{\ms}_z(\bmu) P_{L^2}(\Phi_{z}^{\ms}({\bmu})) \|_{L^2(\Omega)}\\
\label{new-inserted-step}&= \| P_{L^2}(u^{\ms}(\cdot\hspace{2pt};\bmu)) \|_{L^2(\Omega)} \lesssim 
\| \nabla u^{\ms}(\cdot\hspace{2pt};\bmu) \|_{L^2(\Omega)},
\end{align}
where in the last step we used the $L^2$-stability of $P_{L^2}$ and the Poincar\'e-Friedrichs inequality.
With that we obtain
\begin{eqnarray*}
\lefteqn{\inf_{z \in \VmsRB(\bmu)} \| u^{\ms}(\cdot\hspace{2pt};\bmu) - z \|_{H^1(\Omega)}^2 \le
\| \sum_{z \in \mathcal{N_H}} u^{\ms}_z(\bmu) \left( \Phi_{z}^{\ms}({\bmu}) - \Phi_{z}^{\rb}({\bmu}) \right) \|_{H^1(\Omega)}^2} \\
&\lesssim& \sum_{z \in \mathcal{N_H}} k^d |u^{\ms}_z(\bmu)|^2 \|  \nabla  \left( \Phi_{z}^{\ms}({\bmu}) - \Phi_{z}^{\rb}({\bmu}) \right)\|_{L^2(\Omega)}^2
\overset{(\ref{main-theorem-step-1})}{\lesssim} e^{-2\zeta J} k^d \sum_{z \in \mathcal{N_H}} |u^{\ms}_z(\bmu)|^{2}  \| \nabla \Phi_z \|_{L^2(\omega_z)}^{2}\\
&\lesssim& e^{-2\zeta J} H^{-2} k^d \sum_{z \in \mathcal{N_H}} |u^{\ms}_z(\bmu)|^2 \|  \Phi_z \|_{L^2(\omega_z)}^2
\overset{\eqref{new-inserted-step}}{\lesssim} e^{-2\zeta J} H^{-2} k^d \| \nabla u^{\ms}(\cdot\hspace{2pt};\bmu) \|_{L^2(\Omega)}^2 \\
&\lesssim& e^{-2\zeta J} H^{-2} k^d \| f(\cdot,\bmu) \|_{L^2(\Omega)}^2,
\end{eqnarray*}
where $\| f(\cdot,\bmu) \|_{L^2(\Omega)}$ is uniformly bounded by (A2). In total, using Galerkin orthogonality and Proposition \ref{prop-conv-lod} with $k \ge C_{g} |\log(H)|$, we obtain
\begin{align*}
\| \uRB(\cdot\hspace{2pt};\bmu) - u_h(\cdot\hspace{2pt};{\bmu}) \|_{H^1(\Omega)} &\lesssim 
\inf_{z \in \VmsRB(\bmu)}  \| z - u_h(\cdot\hspace{2pt};{\bmu}) \|_{H^1(\Omega)}\\
&\lesssim \inf_{z \in \VmsRB(\bmu)}  \| z - u^{\ms}(\cdot\hspace{2pt};\bmu) \|_{H^1(\Omega)} + \| u^{\ms}(\cdot\hspace{2pt};\bmu) - u_h(\cdot\hspace{2pt};{\bmu}) \|_{H^1(\Omega)}\\
&\lesssim e^{-\zeta J} H^{-1} k^{d/2} + H.
\end{align*}
\end{proof}

\begin{remark}[General boundary conditions]
We note that the LOD (and consequently also the RB-LOD) can be generalized to any type of mixed Dirichlet and Neumann boundary conditions for problem \eqref{equation-strong}. If the boundary condition is basically described by a \quotes{coarse} function, the adaption of the method is straightforward, i.e. all local problems are still solved in the same way as before and the boundary condition is only incorporated in the weak formulation of the final global problem. In particular, we still solve in the same multiscale space $\VmsRB(\bmu)$ that was obtained for the case of a homogenous Dirichlet boundary condition. If the boundary condition is oscillatory and complicated, additional boundary correctors need to be introduced to preserve the old convergence rates in $H$. For further details we refer to \cite{HeM14} where an LOD for general boundary value problems is analyzed.
\end{remark}

\begin{remark}[Complexity]
Let us assume that $\T_H$ and $\T_h$ are quasi-uniform so that the number of coarse nodes is of order $\mathcal{O}(H^{-d})$ and the number of fine nodes of order $\mathcal{O}(h^{-d})$. In this case and for a fixed parameter $\bmu$, the cost for solving all local problems \eqref{local-corrector-problem} associated with a patch $U_k(K)$ with $K \in \mathcal{T}_H$ and $k\simeq |\log(H)|$ are $\mathcal{O}((H |\log(H)|^2 / h)^d)$ (as already discussed before). The cost for solving all local problems are consequently $\mathcal{O}((|\log(H)|^2 / h)^d)$. The global system matrix (constructed from the LOD basis functions) has $\mathcal{O}(k^d H^{-d})$ non-zero entries. This is still sparse, with entries that decay away from the diagonal. Consequently, suitable solvers can increase the computational complexity for solving the arising system to an almost linear complexity $\mathcal{O}((1/H)^d)$, which is negligible compared to the cost for solving the local problems. Hence, the computational complexity for applying the original LOD once for one single parameter are of order $\mathcal{O}((H |\log(H)|^2 / h)^d)$. 

The RB-LOD applied to a parametrized problem increases these cost by a factor that depends on the Greedy procedure. However, this only effects the offline stage. In the online phase, we have to solve $\mathcal{O}((1/H)^d)$ local problems, each of them having the dimension of order $\mathcal{O}(J)$, where $J$ denotes the dimension of the local RB-space (and where the basis of the local space is orthonormal). Hence, the cost for solving all local problems in the online phase are $\mathcal{O}(J(1/H)^d)$. For the same reasons as for the standard LOD, the global problem can be solved with a complexity which is close to $\mathcal{O}((1/H)^d)$. In summary, once the offline computations are finished, the RB-LOD has a complexity of $\mathcal{O}(J(1/H)^d)$ to solve the global parametrized problem, whereas the standard LOD has a complexity of $\mathcal{O}((|\log(H)|^2 / h)^d)$ for the same task.
\end{remark}

\section{Numerical experiments}
\label{section-numerical-experiments}

In this section we present two numerical experiments. In the first numerical experiment, we consider a parameterized linear elliptic problem as given by (\ref{equation-strong}) and demonstrate the applicability of RB-LOD. In the second numerical experiment we show how the method can be used to solve nonlinear elliptic problems such as the stationary Richards equation.

As a measure for the error we will consider the relative error norms
$\| \cdot \|_{L^2(\Omega)}^{\mbox{\tiny rel}}$, respectively $\| \cdot \|_{H^1(\Omega)}^{\mbox{\tiny rel}}$, we denote, i.e. the absolute errors divided by the associated norm of the fine scale reference solution.

For the computation times stated in Table \ref{model-problem-1-cpu-times} and \ref{model-problem-2-cpu-times}, we use the following notation.
\begin{itemize}
\item
 $t^{\mbox{\rm \scriptsize off, local}}(K)$: For a fixed coarse element $K\in \T_H$, $t^{\mbox{\rm \scriptsize off, local}}(K)$ denotes the time for solving all local problems (in the offline phase) that are associated with this element. Note that each coarse element contains 3 coarse nodes and for each coarse node the problem needs to be solved in average for $2-16$ parameters (depending on the model problem and the location) before the error falls below the tolerance. The time also contains the time for solving for the local Riesz representatives (which are required for the error estimation) and it contains the time for assembling the required system matrices and right hand sides. Note that a further parallelization is possible here.
\item $t^{\mbox{\rm \scriptsize off, local}}_{\mbox{\rm \scriptsize average}}$: The average time of $t^{\mbox{\rm \scriptsize off, local}}(K)$ over all $K \in \T_H$.
\item $t^{\mbox{\rm \scriptsize on, local}}_{\mbox{\rm \scriptsize average}}$: The average time for solving one local problem in the online phase. Note, in the online phase we only need to solve one problem for each coarse node.
\item $t^{\mbox{\rm \scriptsize on, global}}_{\mbox{\rm \scriptsize average}}$: The average time for solving the global problem (in the case of model problem 1) or the average time for solving the global problem for one iteration step (in the case of model problem 2).
\end{itemize}

\subsection{Model Problem 1}
In this section we consider the following parametrized model problem: 
find $\ueps(x;{\bmu})\in H^1_0(\Omega)$ such that
\begin{align*}
 - \nabla \cdot \left( \aeps(x;{\bmu}) \nabla \ueps(x;{\bmu}) \right) &= 1 \qquad \mbox{in } \Omega,
 \end{align*}
with
$$\aeps(x;{\bmu}) := \sum_{q =1}^4 \Theta_q({\bmu}) \aeps_q(x),$$
where 
\begin{align*}
  \Theta_1(\bmu) &:= 2 + \sin(4 \bmu ), \hspace{47pt}
  \Theta_2(\bmu) := 2 + \bmu^2 - \cos( \sqrt{|\bmu|} ), \\
  \Theta_3(\bmu) &:= 2 + \cos( \sqrt{|\bmu|} ) \quad \mbox{and} \quad
  \Theta_4(\bmu) := 1 + \sqrt{|\bmu|} + (1/10) |\bmu|^{3/2}.
\end{align*}
and $\aeps_q$ functions given by
$\aeps_1$, $\aeps_2$, $\aeps_3$ and $\aeps_4$. For a given $\eps=0.1$ we define
\begin{align*}
\aeps_1(x_1,x_2) &:= \left(
\begin{matrix}
5 \pi^{-2} \left( 4 + 2 \cos( 2 \pi x_1 / \eps ) \right)^{-1} & 0 \\
 0 & (4 \pi)^{-1} \left( 5 +  2.5 \cos( 2 \pi x_1 / \eps ) \right)
\end{matrix} \right),\\
\aeps_2(x_1,x_2) &:= \left( 10 + 9 \sin(2 \pi \sqrt{2 x_1} /\eps ) \sin( 4.5 \pi x_2^2 / \eps) \right) \left(
\begin{matrix}
1/100 & 0 \\
0 & 1/100 
\end{matrix} \right),\\
\aeps_3(x) &:= \left( (3/25) + (1/20) g^{\eps}(x) \right) \left(
\begin{matrix}
1 & 0 \\
0 & 1
\end{matrix} \right),
\end{align*}
where
\begin{align*}
g^{\eps}(x_1,x_2) := \sin\left( \lfloor x_1 + x_2 \rfloor + \lfloor \frac{x_1}{\eps} \rfloor +  \lfloor \frac{x_2}{\eps} \rfloor\right) + \cos\left( \lfloor x_2 - x_1 \rfloor +  \lfloor \frac{x_1}{\eps} \rfloor +  \lfloor \frac{x_2}{\eps} \rfloor \right),
\end{align*}
and we define the last coefficient by
\begin{align*}
\aeps_4(x)&:=(h \circ c_\varepsilon)(x)
\qquad \text{with} \enspace h(t):=\begin{cases}
t^4 &\text{for} \enspace \frac12 < t < 1 \\ 
t^{\frac{3}{2}} &\text{for} \enspace 1 < t < \frac{3}{2}  \\ 
t &\text{else}
\end{cases}
\end{align*}
and
\begin{displaymath}
c^\eps(x_1,x_2):=1 + \frac{1}{10} \sum_{j=0}^4 \sum_{i=0}^{j} \left( \frac{2}{j+1} \cos \left( \bigl\lfloor i x_2 - \tfrac{x_1}{1+i} \bigr\rfloor + \left\lfloor \tfrac{i x_1}\varepsilon \right\rfloor + \left\lfloor \tfrac{ x_2}\varepsilon \right\rfloor \right) \right).
\end{displaymath}

We set  $\mathcal{D}:=[0,5]$ and a training set consisting of $100$ randomly distributed parameters. The online parameter is chosen to be ${\bmu}=2.012$.

In all the computations the fine grid $\T_h$ is the uniformly refined grid with mesh size $h=2^{-7}$ which resolves the microstructure. We also fixed the relative RB tolerance for the offline Greedy search (i.e. Step \ref{step-4}) with the value TOL$=0.1$. This value is small enough so that we recover the classical convergence rates for the RB-LOD solution (for the selected online parameter and in the sense of Proposition \ref{prop-conv-lod}) and smaller values do not change the final errors. {The RB-LOD approximation is denoted by $\uRB$ and the reference solution obtained with a standard finite element method on the fine grid $\T_h$ is denoted by $u_h$.} The results are depicted in Table \ref{mod-prob-1-errors} and \ref{model-problem-1-table-EOCs}. More precisely, in the online phase we observe a linear convergence (with respect to the coarse mesh size $H$) for the $H^1$-error and a quadratic order convergence for the $L^2$-error. We also observe a linear order convergence of the coarse part of the {RB-LOD approximation 
denoted by $u_H^c$ and given by the $L^2$-projection of the RB-LOD approximation into the coarse space, i.e.
$u_H^c:= \sum_{z\in \mathcal{N}_H} \alpha_z^{\rb} \Phi_z$ if $\uRB = \sum_{z\in \mathcal{N}_H} \alpha_z^{\rb} \Phi_{z}^{\rb}$. In some cases, e.g. if we are only interested in an effective $L^2$-approximation of the exact solution, it might be hence sufficient to only store $u_H^c$ instead of $\uRB$ (see also \cite{EGH15}).}

The findings are emphasized by Figure \ref{mp-1-rb-lod-basis-funcs}, where the various approximations are depicted for the case $(H,k) = (2^{-3},2)$. Despite the very coarse coarse grid, we still observe that the RB-LOD approximation is hardly distinguishable from the FEM reference solution for $h=2^{-7}$. Even the coarse part alone captures all relevant features.

The number of local RB-LOD basis functions (i.e. the dimension of the space $\WzRB$ computed in the offline Step \ref{step-4}) is between $4$ and $16$ depending on the node $z$. In average $\WzRB$ contains between $7-8$ basis functions. Hence its dimension is very small. 
An example for a particular node $z\in \mathcal{N}_H$ is given in Table \ref{model-problem-1-exp-decay-in-rb-space}, where we also illustrate the decay of the RB error.
In the online phase (i.e. in Step \ref{online-step}) we use the spaces $\WzRB$ to compute {\it one} online basis function $\Phi_{z}^{\rb}({\bmu})$ for each coarse node $z\in\mathcal{N}_H$. The RB-LOD online basis function $\Phi_{z}^{\rb}({\bmu})$ is a smoothed version of the classical coarse nodal basis function $\Phi_z$. An example how it looks like is given in Figure \ref{mp-1-rb-lod-basis-funcs}.

Finally, the computational costs are depicted in Table \ref{model-problem-1-cpu-times}. As expected, the costs in the online phase are extremely low. The offline costs that arise for a single coarse element $K\in \T_H$ (see first column of Table \ref{model-problem-1-cpu-times}) might appear a bit high at first glance. However, recall that we need to solve a lot of problems on each of the coarse elements. Since we have in average $8$ relevant parameters for each coarse node, since $K$ contains $3$ coarse nodes and since we need to solve the corrector problems and the problems for the Riesz representatives, we need to solve approximately $3 \times 8 \times 2=48$ local saddle point problems in each coarse element $K$. This explains the high CPU times. However, we also note that these problems (associated with an element $K$) can be further parallelized. Hence, if enough cores are available the costs can be distributed and the CPU times for one element can be even decreased to a few seconds.

\begin{table}[h!]
\caption{\it Model Problem 1. For $h=2^{-7}$ we denote the full RB-LOD error by $\eRB:=\pRB-p_h$ and the coarse part of the RB-LOD error by $\eH:=\uRB - u_h$. By $k$ we denote the (fixed) localization according to Definition \ref{definition:localized:ms:space}. The table depicts errors for various combinations of $H$ and $k$.}
\label{mod-prob-1-errors}
\begin{center}
\begin{tabular}{|c|c|c|c|c|c|c|c|}
\hline $H$      & $k$ 
& $\| \eH \|_{L^2(\Omega)}^{\mbox{\tiny rel}}$
& $\| \eRB \|_{L^2(\Omega)}^{\mbox{\tiny rel}}$
& $\| \eRB \|_{H^1(\Omega)}^{\mbox{\tiny rel}}$ \\
\hline
\hline $2^{-2}$ & 0   & 0.20386 & 0.20355 & 0.44348 \\
\hline $2^{-2}$ & 1   & 0.11778 & 0.05939 & 0.23663 \\
\hline $2^{-2}$ & 2   & 0.11331 & 0.04106 & 0.14681 \\
\hline $2^{-2}$ & 3   & 0.11325 & 0.03373 & 0.12434 \\
\hline
\hline $2^{-3}$ & 1   & 0.14199 & 0.14207 & 0.39250 \\
\hline $2^{-3}$ & 2   & 0.02878 & 0.01076 & 0.09263 \\
\hline $2^{-3}$ & 3   & 0.02862 & 0.00869 & 0.07865 \\ 
\hline
\hline $2^{-4}$ & 1   & 0.49456 & 0.49483 & 0.69713 \\
\hline $2^{-4}$ & 2   & 0.01137 & 0.00749 & 0.09558 \\
\hline $2^{-4}$ & 3   & 0.00984 & 0.00663 & 0.07746 \\ 
\hline $2^{-4}$ & 4   & 0.00976 & 0.00172 & 0.02834 \\
\hline
\end{tabular}\end{center}
\end{table}

\begin{table}[h!]
\caption{\it Model Problem 1. As suggested by Proposition \ref{prop-conv-lod} we couple $k$ and $H$ by $k=k(H):=\lfloor |\log(H)| + 1 \rfloor$ to recover the typical convergence rates. The full RB-LOD error is given by $\eRB:=\pRB-p_h$ and the coarse part by $\eH:=\uRB - u_h$. 
For each of the errors $\|e_H\|$ below (for $H=2^{-i}$), we define the average EOC by EOC$:= \frac{1}{2} \sum_{i=1}^2 \log_2(\| e_{2^{-i}} \| / \| e_{2^{-(i+1)}} \|)/\log_2(2)$.}
\label{model-problem-1-table-EOCs}
\begin{center}
\begin{tabular}{|c|c|c|c|c|c|c|c|}
\hline $H$      & $k(H)$
& $\| \eH \|_{L^2(\Omega)}^{\mbox{\tiny rel}}$
& $\| \eRB \|_{L^2(\Omega)}^{\mbox{\tiny rel}}$
& $\| \eRB \|_{H^1(\Omega)}^{\mbox{\tiny rel}}$ \\
\hline
\hline $2^{-2}$ & 2   & 0.11331 & 0.04106 & 0.14681 \\
\hline $2^{-3}$ & 3   & 0.02862 & 0.00869 & 0.07865 \\ 
\hline $2^{-4}$ & 4   & 0.00976 & 0.00172 & 0.02834 \\
\hline
\hline \multicolumn{2}{|c|}{EOC}  & 1.769 & 2.289  & 1.187 \\
\hline
\end{tabular}\end{center}
\end{table}

\begin{table}[h!]
\caption{\it Model Problem 1. 
Let $\T_h$ denote the (uniformly refined) fine grid with $h=2^{-7}$. {The results in the table refer to a given coarse node $z=(0.03125,0.03125)\in \mathcal{N}_H$. Using the Greedy algorithm (i.e. Step \ref{step-4}) we identify the relevant parameters that are added to the local parameter set $\Xi_{z}^{\rb}$. In the first column we depict the added parameter and in the second column the corresponding maximum estimated error. In our example, 4 parameters are added before the maximum estimated error falls below the tolerance TOL$=0.1$.}}
\label{model-problem-1-exp-decay-in-rb-space}
\begin{center}
\begin{tabular}{|c|c|}
\hline added parameter & maximum estimated error \\
\hline
\hline $2.5$ & 1.47381 \\
\hline $5.0$ & 0.60277 \\
\hline $0.353535$ & 0.10932 \\
\hline $ 4.39394$  & $<$ TOL \\
\hline
\end{tabular}
\end{center}
\end{table}

\begin{figure}[h!]
\centering
\includegraphics[scale=0.18]{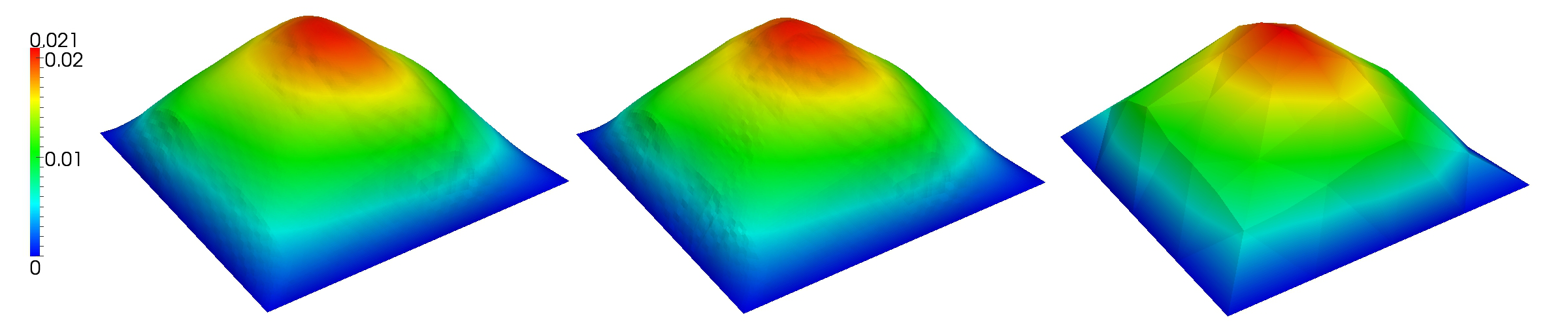}
\caption{\it Model Problem 1. Left Picture: Standard finite element approximation on $\T_h$ with $h=2^{-7}$. Middle picture: RB-LOD approximation (for the online parameter $\bmu=2.012$) on $\T_H$ with $H = 2^{-3}$ and for $(k,h)=(2,2^{-7})$. Right picture: The coarse part of the aforementioned RB-LOD approximation, i.e. its $L^2$-Projection in $V_H$.}
\label{mp-1-rb-serie-warps}
\end{figure}

\begin{figure}[h!]
\centering
\includegraphics[scale=0.22]{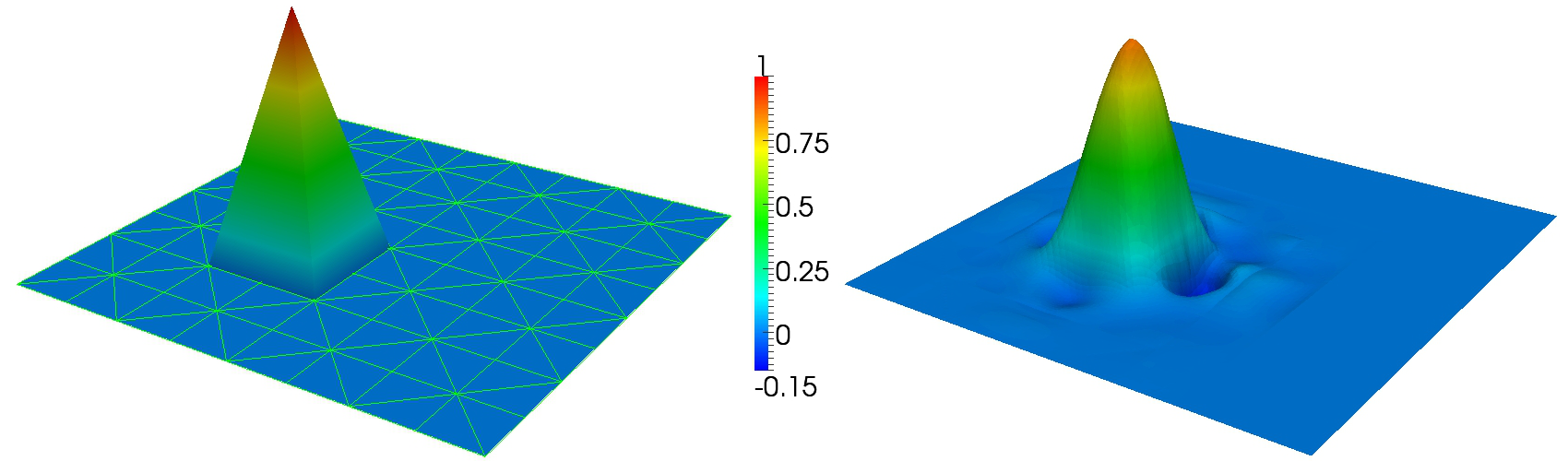}
\caption{\it Model Problem 1. Left Picture: Nodal basis function on the coarse grid $\T_H$ with $H = 2^{-3}$. Right Picture: The corresponding (online) RB-LOD basis function (on $\T_H$ with $H = 2^{-3}$) determined with the described RB-LOD method with localization parameter $k=2$ and $h=2^{-7}$. The basis function was determined for the online parameter $\bmu=2.012$.}
\label{mp-1-rb-lod-basis-funcs}
\end{figure}

\begin{table}[h!]
\caption{\it Model Problem 1. The table depicts various CPU times where we use the notation introduced at the beginning of this section.}
\label{model-problem-1-cpu-times}
\begin{center}
\begin{tabular}{|c|c|c|c|c|c|c|c|}
\hline $H$      & $k$ 
& $t^{\mbox{\rm \scriptsize off, local}}_{\mbox{\rm \scriptsize average}}$
& $t^{\mbox{\rm \scriptsize on, local}}_{\mbox{\rm \scriptsize average}}$
& $t^{\mbox{\rm \scriptsize on, global}}_{\mbox{\rm \scriptsize average}}$ \\
\hline
\hline $2^{-2}$ & 2   & 343.17 [s]  & 0.15   [s] & 0.003 [s] \\
\hline $2^{-3}$ & 2   & 119.29 [s]  & 0.132 [s] & 0.003 [s] \\
\hline $2^{-4}$ & 2   & 38.30 [s]    & 0.136 [s] & 0.004 [s] \\
\hline
\end{tabular}\end{center}
\end{table}

\subsection{Model Problem 2}
\label{ss-mod-prob-2}

In model problem 2 we consider the stationary Richards equation
given in the stationary case by
\begin{equation}
\label{richards-equation}
-\nabla \cdot \left( K^{\eps}(x) kr(x,s(x))) \nabla p(x) \right) = f(x) \qquad \mbox{in} \enspace \Omega.
\end{equation}

The Richards equation has two unknowns, the pressure $p$ and the saturation/water content $s$. It
describes the distribution of subsurface water and can be used for simulating flooding events or to predict the effects of dams or modifications of river courses. The Richards equation can be derived from the two-phase flow equations (with water as the first phase and air as the second) under the assumption that the pressure in the second phase is basically constant. Here, $K^{\eps}$ describes the absolute permeability, $kr$ the (soil-type- and saturation-dependent) relative permeability and $f$ a given source or sink term. The saturation $s$ takes values between $0$ and $1$, where $0$ means that a region is completely dry and $1$ means that the region is fully occupied by water. To remove one of the unknowns from the equation, it is possible to find a model that expresses the pressure $p$ in terms of the saturation $s$. The most popular models are according to Brooks and Corey \cite{BrC64}, Van Genuchten \cite{vGe80} and Gardner \cite{Gar58}. In the following example, we use the Brooks-Corey model which is given in the following way. Assume that the domain $\Omega$ consists of a union of subdomains $\Omega_q$ (for $q \in \mathcal{Q}$), where each of the $\Omega_q$ is occupied by a different type of soil. Then (cf. see \cite{BKS11}) for $x \in \Omega_q$, we can approximate the saturation by the (Brooks-Corey) pressure-saturation curve {$\theta_q$ with $\theta_q(p(x))=s(x)$ and} that is given by
\begin{align}
\label{water-content}\theta_q(p):= \begin{cases}
\theta_m^{(q)} + (\theta_M^{(q)} - \theta_m^{(q)}) \left( \frac{p}{p_b^{(q)}} \right)^{-\lambda^{(q)}} \quad &\mbox{for} \enspace p \le p_b^{(q)},\\
\theta_M^{(q)} \quad &\mbox{for} \enspace p \ge p_b^{(q)}.
\end{cases}
\end{align}
Here, $\theta_m^{(q)},\theta_M^{(q)} \in [0,1]$, $p_b^{(q)} <0$ and $\lambda^{(q)} >0$ are soil dependent parameters. The parameter $\theta_m^{(q)}$ denotes the residual saturation, $\theta_M^{(q)}$ the maximal saturation, $p_b^{(q)}$ the bubbling pressure and $\lambda^{(q)}$ the pore size distribution factor. The total relative permeability $kr(\theta)$ is given by
$$kr(x,\theta) := \sum_{q \in \mathcal{Q}} \chi_{\Omega_q}(x) kr_q(\theta),$$
where $\chi_{\Omega_q}$ is the indicator function of $\Omega_q$ and $kr_q(\theta)$ is given by
\begin{align}
\label{relative-permeability}kr_q(\theta) := \left( \frac{\theta - \theta_m^{(q)}}{\theta_M^{(q)} -\theta_m^{(q)}}\right)^{3+\frac{2}{\lambda^{(q)}}}, \quad \mbox{for} \enspace \theta \in [\theta_m^{(q)},\theta_M^{(q)}].
\end{align}
The (possibly rapidly varying) absolute permeability on $\Omega_q$ is given by $K^{\eps}_q$. In total, we can define $\Theta_q( p ):= (kr_q \circ \theta_q)( p )$, $\aeps(x,\bmu):=\sum_{q \in \mathcal{Q}} \chi_{\Omega_q}(x) K^{\eps}_q(x) \Theta_q( \bmu )$ and search for $p$ satisfying
\begin{align}
\label{richards-equation}-\nabla \cdot \left( \aeps(\cdot, p) \nabla p \right) = f \qquad \mbox{in} \enspace \Omega.
\end{align}

\begin{figure}[h!]
\centering
\begin{picture}(100,100)
\linethickness{1.0pt}
\put(0,0){\line(1,0){100}}
\put(0,0){\line(0,1){100}}
\put(100,0){\line(0,1){100}}
\put(0,100){\line(1,0){100}}
\linethickness{1.5pt}
\put(0,0){{\color{red}\dottedline{4}(0,60)(60,60)}}
\put(0,0){{\color{red}\dottedline{4}(60,0)(60,60)}}
\put(0,0){{\color{blue}\dottedline{4}(40,60)(100,60)}}
\put(0,0){{\color{blue}\dottedline{4}(40,0)(40,60)}}
\put(0,0){{\color{green}\dottedline{4}(0,40)(60,40)}}
\put(0,0){{\color{green}\dottedline{4}(60,40)(60,100)}}
\put(0,0){{\color{black}\dottedline{4}(40,40)(100,40)}}
\put(0,0){{\color{black}\dottedline{4}(40,40)(40,100)}}
\put(10,25){{\color{red}$\boldsymbol\Omega_1$}}
\put(5,15){{\color{red}sandy}}
\put(5,5){{\color{red}soil}}
\put(75,25){{\color{blue}$\boldsymbol\Omega_2$}}
\put(70,15){{\color{blue}sand}}
\put(10,88){{\color{dark-green}$\boldsymbol\Omega_3$}}
\put(5,78){{\color{dark-green}sandy}}
\put(5,68){{\color{dark-green}loam}}
\put(73,88){{\color{black}$\boldsymbol\Omega_4$}}
\put(68,78){{\color{black}loamy}}
\put(68,68){{\color{black}sand}}
\end{picture}
\caption{\it Illustration of the overlapping decomposition of $\Omega$ into the subdomains $\Omega_1$, $\Omega_2$, $\Omega_3$ and $\Omega_4$.}
\label{mp-2-omega-decomposed}
\end{figure}
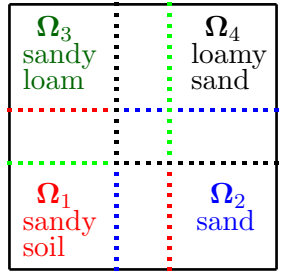

\begin{table}[h!]
\caption{\it The table depicts soil parameters for different soil types.  We let $\theta_m^{(q)}$ denote the residual- and $\theta_M^{(q)}$ the maximal saturation, furthermore $p_b^{(q)}$ denotes the bubbling pressure and $\lambda^{(q)}$ the pore size distribution factor. The values are taken from \cite{RAB93}.}
\label{table-soil-parameters}
\begin{center}
\begin{tabular}{|c||c||c|c|c|c|}
\hline $q$ & Soil type & $\theta_m^{(q)}$ & $\theta_M^{(q)}$ & $\lambda^{(q)}$ & $p_b^{(q)}$ \\
\hline
\hline 1 & sandy soil    & 0.21     & 0.95 & 1.0      & -0.1 [m] \\
\hline 2 & sand            & 0.0458 & 1.0   & 0.694  & -0.0726 [m] \\
\hline 3 & sandy loam & 0.091   & 1.0   & 0.378  & -0.147 [m] \\
\hline 4 & loamy sand & 0.08     & 1.0   & 0.553  & -0.087 [m] \\
\hline
\end{tabular}\end{center}
\end{table}

As a specific model problem realization, we consider the stationary Richards-equation (\ref{richards-equation}) with a homogeneous Dirichlet boundary condition and $f \equiv 1$ and set $\eps := 0.1$. We let $\Omega:= ]-1,1[^2$ be union of the slightly overlapping subdomains $\Omega_1,\ldots,\Omega_4$ that are given by
\begin{align*}
\Omega_1 &:= [ 0, 1/2 + \eps ] \times [ 0, 1/2 + \eps ], \quad 
\Omega_2 := [ 1/2 - \eps, 1 ] \times [ 0, 1/2 + \eps ],\\
\Omega_3 &:= [ 0, 1/2 + \eps ] \times [ 1/2 - \eps, 1 ], \quad 
\Omega_4 := [ 1/2 - \eps, 1 ] \times [ 1/2 - \eps, 1 ].
\end{align*}
The domain $\Omega_1$ is occupied by sandy soil, $\Omega_2$ by sand, $\Omega_3$ by sandy loam and $\Omega_4$ by loamy sand (see Figure \ref{mp-2-omega-decomposed}). For $q=1,\ldots,4$, we pick $K^{\eps}_q:=\aeps_q$, where $\aeps_q$ is given as introduced at the beginning of this section. The relative permeabilities $kr_q(\theta)$ and the water contents $\theta_q(p)$ are given according to the equations (\ref{relative-permeability}) and (\ref{water-content}, where the corresponding soil parameters are stated in Table \ref{table-soil-parameters}.

We aim to solve \eqref{richards-equation} with the RB-LOD. In the first step, we perform the offline preprocessing as described in Sections \ref{subsection-initialization} - \ref{subsection-precomputation-local-entries} for $\aeps(\cdot,p)$ as above. We note that the solution itself enters into the parameter set but fortunately only upper and lower bound are needed to set a parameter range of {\it possible} pressure to proceed with the RB algorithm.
The upper bound for the compact parameter set $\mathcal{D} \subset \R$ is naturally given by the maximum bubbling pressure (i.e. -0.0726 in our example). Theoretically, there exists no lower bound since the pressure $p$ can fall to $-\infty$, however practically we observe that $\theta_q(p)$ converges quickly (with order $\lambda^{(q)}$) to $\theta_m^{(q)}$ (cf. \cite{BKS11}). Hence we can use a small negative integer for the lower bound in $\mathcal{D}$. In our numerical experiments we picked $\mathcal{D}:=[-2,-0.0726]$ and a training set consisting of $100$ randomly distributed parameters in $\mathcal{D}$. Larger sets (for $\mathcal{D}$ and the training set) are possible, but not necessary in our example. The relative RB tolerance for the offline Greedy search (i.e. Step \ref{step-4}) was set to TOL$=0.01$ for our experiments. Depending on the location, the algorithm identified between $2$ and $10$ relevant parameters for each coarse node $z \in \mathcal{N}_H$, before the tolerance was reached. 

After assembling the local RB-LOD basis function sets and the corresponding entries for the local system matrix and the source term (i.e. the steps described in Section \ref{subsection-precomputation-local-entries}), we can perform the global online computation that we describe in the following. In order to solve the Richards equation we apply the Newton method. Recall that $N$ denotes the dimension of the coarse space $V_H$ (and hence it is also the dimension of the RB space $\VmsRB$). Let the initial value $\mathbf{p_0}=\left( \bf{p}_0^1, \ldots, \bf{p}_0^N \right) \in \R^N$ be given, where each entry $\bf{p}_0^z$ is associated with a coarse node $z \in \mathcal{N}_H$ ($\mathbf{z}$ denoting the index of $z$). The first online RB space $\VmsRB({\bf{p}_0})$ is assembled by using {$\bmu^0(z)=\bf{p}_0^z$} as the online parameter in Step \ref{online-step} (i.e. the online parameter is not globally fixed, but varies for each coarse node).

Assuming that the solution $\bf{p}_n \in \R^N$ of the $n$'th Newton step is computed, we assemble the corresponding global RB space $\VmsRB({\bf{p}_{n}})$ again according to Step \ref{online-step} (where ${\bf{p}_{n}^z}$ is the online parameter for node $z$), i.e.
$\VmsRB({\bf{p}_{n}}):=\mbox{\rm span}\{  \Phi_{z}^{\rb}({\bf{p}_{n}^z})| \hspace{2pt} z \in \mathcal{N}_H\}.$
With that, we define
\begin{align}
\label{formula-prbn-mod-prob-2}\pRBn := \sum_{z \in \mathcal{N}_H} {\bf{p}_{n}^z} \hspace{2pt} \Phi_{z}^{\rb}({\bf{p}_{n}^z}) \hspace{4pt} \in \VmsRB({\bf{p}_{n}})
\end{align}
as the current LOD approximation. To update $\pRBn$ we perform the classical Newton step:
find  $\deltaRBnn=\sum_{z \in \mathcal{N}_H} \boldsymbol{\delta}_{p,n} \Phi_{z}^{\rb}({\bf{p}_{n}^z}) \in \VmsRB({\bf{p}_{n}})$
such that
\begin{eqnarray}
\label{newton-step-rb}\nonumber\lefteqn{\int_{\Omega} A^{\eps}(\cdot, \pRBn) \nabla \deltaRBnn \cdot \nabla v
+ \int_{\Omega} \left( D_2 A^{\eps}(\cdot, \pRBn) \nabla \pRBn \cdot \nabla v \right) \deltaRBnn}\\
&=& \int_{\Omega} f v - \int_{\Omega} A^{\eps}(\cdot, \pRBn) \nabla \pRBn \cdot \nabla v \qquad \mbox{for all } v \in \VmsRB({\bf{p}_{n}})
\end{eqnarray}
where $D_2 A^{\eps}(\cdot, \pRBn)=\partial/\partial p A^{\eps}(\cdot, \pRBn)$.
Solving the corresponding linear system gives $\boldsymbol{\delta}_{p,n} \in \R^N$ 
and one can update the solution coefficient vector by
\begin{align}
\label{formula-prbn-mod-prob-2-eq-2}\mathbf{p}_{n+1}:=\mathbf{p}_n+ \boldsymbol{\delta}_{p,n}.
\end{align}

The Newton algorithm stops, when the norm $\boldsymbol{\delta}_{p,n}$ falls below a given tolerance. The only problem with this approach is that it requires quadrature costs in each iteration step in order to assemble the system matrix. It is not directly possible to precompute certain entries in an offline phase. However, it is possible to introduce an additional simplification. We make the following consideration.

\begin{remark}[Numerical quadrature for $A^{\eps}(\cdot, \pRBn)$ and $D_2 A^{\eps}(\cdot, \pRBn)$]
\label{remark-num-quad}We wish to simplify equation (\ref{newton-step-rb}) by approximating it with a numerical quadrature rule, since the entries of the global system matrix
cannot be  pre-computed at the moment. First, recall the definition of $\pRBn$ stated in equation (\ref{formula-prbn-mod-prob-2}). Since $\Phi_{z}-\Phi_{z}^{\rb}({\bf{p}_{n}^z})\in W_h$ we have $P_{L^2}(\Phi_{z}^{\rb}({\bf{p}_{n}^z}))=\Phi_z$ for all time steps $n$ and independent of ${\bf{p}_{n}^z}$. Recalling Remark \ref{remark-on-L2-projection}, we conclude that
\begin{align*}
\pHn := P_{L^2}(\pRBn) = \sum_{z \in \mathcal{N}_H} {\bf{p}_{n}^z} \hspace{2pt} \Phi_{z} \hspace{4pt} \in V_H
\end{align*}
and hence
$$\| \pHn - \pRBn \|_{L^2(\Omega)} \lesssim H \| \pRBn \|_{H^1(\Omega)},$$
where $\| \pRBn \|_{H^1(\Omega)}$ can typically be further bounded by data functions. Therefore we expect to make an $\mathcal{O}(H)$-error when replacing $\pRBn$ by $\pHn$ in the {\it left hand side} of equation (\ref{newton-step-rb}), i.e.
\begin{eqnarray*}
\lefteqn{\int_{\Omega} A^{\eps}(\cdot, \pRBn) \nabla \deltaRBnn \cdot \nabla v + \int_{\Omega} \left( D_2 A^{\eps}(\cdot, \pRBn) \nabla \pRBn \cdot \nabla v \right) \deltaRBnn}\\
&\approx&
\underset{=:\mbox{I}}{\underbrace{\int_{\Omega} A^{\eps}(\cdot, \pHn) \nabla \deltaRBnn \cdot \nabla v}} + \underset{=:\mbox{II}}{\underbrace{\int_{\Omega} \left( D_2 A^{\eps}(\cdot, \pHn) \nabla \pRBn \cdot \nabla v \right) \deltaRBnn}}
+  \mathcal{O}(H).
\end{eqnarray*}
Note that we are not allowed to replace $\nabla \pRBn$ by $\nabla \pHn$. 
To illustrate the simplification that this substitution yields, we consider the first term $\mbox{I}$ (the second term can be treated in a similar way). For $v=\Phi_{y}^{\rb}(\mu(y)^{n},\cdot),$ 
where $y \in \mathcal{N}_H$
by using the affine decomposition of $A^{\eps}$ we see that
\begin{eqnarray*}
\lefteqn{\int_{\Omega} A^{\eps}(x, \pHn(x)) \nabla \deltaRBnn(x) \cdot \nabla \Phi_{y}^{\rb}({\bf{p}_{n}^y},x) \hspace{2pt} dx}\\
&=& \sum_{q=1}^4 \sum_{z \in \mathcal{N}_H} \boldsymbol{\delta}_{p,n} \int_{\Omega}  \aeps_q(x) \Theta_q( \pHn(x) ) \nabla \Phi_{z}^{\rb}({\bf{p}_{n}^z},x) \cdot \nabla \Phi_{y}^{\rb}({\bf{p}_{n}^y},x) \hspace{2pt} dx\\
&\approx& \sum_{q=1}^4 \sum_{z \in \mathcal{N}_H} \boldsymbol{\delta}_{p,n} \Theta_q( {\bf{p}_{n}^z} ) \int_{\Omega}  \aeps_q(x) \nabla \Phi_{z}^{\rb}({\bf{p}_{n}^z},x) \cdot \nabla \Phi_{y}^{\rb}({\bf{p}_{n}^y},x) \hspace{2pt} dx + \mathcal{O}(H |\log(H)|).
\end{eqnarray*}
In the last step we used that the integrals are only integrals over $\mbox{supp}(\Phi_{z}^{\rb})$, which has diameter that scales like $H |\log(H)|$. Hence, replacing $\Theta_q( \pHn(x) )$ by $\Theta_q( \pHn(z) )=\Theta_q( {\bf{p}_{n}^z} )$ is expected to only lead to an $\mathcal{O}(H |\log(H)|)$-error. With these quadrature-like modifications, we can precompute an approximation of term $\mbox{I}$ similarly as described in \eqref{entries-global-stiffness-matrix-1}-\eqref{entries-global-stiffness-matrix-2}. For the term $\mbox{II}$ we can proceed the same way. In total, we can derive an approximative formulation for the left hand side in equation (\ref{newton-step-rb}) which allows for a precomputation of global system matrix entries. For the sake of a higher accuracy, we leave the right hand side of (\ref{newton-step-rb}) unaltered. This is unproblematic, since the local quadrature costs are
cheaper and only scale linearly with the number of coarse nodes.
\end{remark}

\begin{table}[h!]
\caption{\it Model Problem 2. For $h=2^{-6}$, we denote the full RB-LOD error by $\eRB:=\pRB-p_h$ and the coarse part of the RB-LOD error by $\eH:=\pH - p_h$. By $k$ we denote the (fixed) localization according to Definition \ref{definition:localized:ms:space}. The table depicts errors for various combinations of $H$ and $k$.}
\begin{center}
\begin{tabular}{|c|c|c|c|c|c|c|c|}
\hline $H$      & $k$ 
& $\| \eH \|_{L^2(\Omega)}^{\mbox{\tiny rel}}$
& $\| \eRB \|_{L^2(\Omega)}^{\mbox{\tiny rel}}$
& $\| \eRB \|_{H^1(\Omega)}^{\mbox{\tiny rel}}$ \\
\hline
\hline $2^{-2}$ & 0   & 0.2523  & 0.2518  & 0.4888 \\
\hline $2^{-2}$ & 1   & 0.1177  & 0.0604  & 0.2167 \\
\hline $2^{-2}$ & 2   & 0.1164  & 0.0654  & 0.2112 \\
\hline $2^{-2}$ & 3   & 0.1163  & 0.0596  & 0.1977 \\ 
\hline
\hline $2^{-3}$ & 1   & 0.2958  & 0.2945  & 0.4592 \\
\hline $2^{-3}$ & 2   & 0.0497  & 0.0204  & 0.1193 \\
\hline $2^{-3}$ & 3   & 0.0494  & 0.0147  & 0.0872 \\ 
\hline
\hline $2^{-4}$ & 1   & 0.5212 & 0.5212 & 0.6106 \\
\hline $2^{-4}$ & 2   & 0.0327 & 0.0253 & 0.1325 \\
\hline $2^{-4}$ & 3   & 0.0222 & 0.0047 & 0.0513 \\
\hline $2^{-4}$ & 4   & 0.0220 & 0.0032 & 0.0363 \\
\hline
\end{tabular}\end{center}
\label{mod-prob-2-errors}
\end{table}

\begin{table}[h!]
\caption{\it Model Problem 2. According to the classical result stated in Proposition \ref{prop-conv-lod} we couple $k$ and $H$ by $k=k(H):=\lfloor |\log(H)| + 0.5 \rfloor$ to recover the typical convergence rates. The full RB-LOD error is given by $\eRB:=\pRB-p_h$ and the coarse part by $\eH:=\pH - p_h$. For each of the errors $\|e_H\|$ below (for $H=2^{-i}$), we define the average EOC by EOC$:= \frac{1}{2} \sum_{i=1}^2 \log_2(\| e_{2^{-i}} \| / \| e_{2^{-(i+1)}} \|)/\log_2(2)$.}
\label{model-problem-2-table-EOCs}
\begin{center}
\begin{tabular}{|c|c|c|c|c|c|c|c|}
\hline $H$      & $k(H)$
& $\| \eH \|_{L^2(\Omega)}^{\mbox{\tiny rel}}$
& $\| \eRB \|_{L^2(\Omega)}^{\mbox{\tiny rel}}$
& $\| \eRB \|_{H^1(\Omega)}^{\mbox{\tiny rel}}$ \\
\hline
\hline $2^{-2}$ & 1   & 0.1177  & 0.0604  & 0.2167 \\
\hline $2^{-3}$ & 2   & 0.0497  & 0.0204  & 0.1193 \\
\hline $2^{-4}$ & 3   & 0.0222  & 0.0047  & 0.0513 \\
\hline
\hline \multicolumn{2}{|c|}{EOC}  & 1.203 & 1.842  & 1.039 \\
\hline
\end{tabular}\end{center}
\end{table}

The subsequent results are obtained for the case {\it without} a numerical quadrature
in the sense of Remark \ref{remark-num-quad}. The method that we used is directly based on the non-modified equations (\ref{formula-prbn-mod-prob-2})-(\ref{formula-prbn-mod-prob-2-eq-2}). However, we note that we implemented both versions of the method and the obtained results were basically the same up to small relative errors of order $10^{-4}$ or less. The relative tolerance for the Newton algorithm to abort was set to $10^{-5}$ (which is sufficiently small to not influence the order of accuracy of the method). For large enough iteration steps $n$ (so that the Newton algorithm aborts) we denote $\pRB:=\pRBn$ and $\pH:=\pHn$. The reference solution (i.e. the solution in the full fine scale finite element space $V_h$) is denoted by $p_h$. The corresponding errors are depicted in Table \ref{mod-prob-2-errors}.
We see that the RB strategy still preserves the behavior of the classical method. We get a fast decay  in terms of the localization parameter $k$ and if we couple $H$ and $k$ according to Proposition \ref{prop-conv-lod} we also recover the classical convergence rates (see Table \ref{model-problem-2-table-EOCs}).

\begin{table}[h!]
\caption{\it Model Problem 2. The table depicts various CPU times.}
\label{model-problem-2-cpu-times}
\begin{center}
\begin{tabular}{|c|c|c|c|c|c|c|c|}
\hline $H$      & $k$ 
& $t^{\mbox{\rm \scriptsize off, local}}_{\mbox{\rm \scriptsize average}}$
& $t^{\mbox{\rm \scriptsize on, local}}_{\mbox{\rm \scriptsize average}}$
& $t^{\mbox{\rm \scriptsize on, global}}_{\mbox{\rm \scriptsize average}}$ \\
\hline
\hline $2^{-2}$ & 2   & 78.15 [s]  & 0.036 [s]  & 0.003 [s] \\
\hline $2^{-3}$ & 2   & 19.12 [s]  & 0.039 [s]  & 0.003 [s] \\
\hline $2^{-4}$ & 2   &   4.71 [s]  & 0.022 [s]  & 0.004 [s] \\
\hline
\end{tabular}\end{center}
\end{table}

Table \ref{model-problem-2-cpu-times} shows that the main computational costs for solving the local problems takes place in the offline phase. Once the the local problems are solved, 
{the computation of the {RB-LOD basis functions} in the online phase (for a given iteration step or a new source term) is very fast.}
The time for computing an online RB-LOD basis function for a given coarse node is of order $0.03$ seconds and hence very small. {The term $t^{\mbox{\rm \scriptsize on, local}}_{\mbox{\rm \scriptsize average}}$ in Table \ref{model-problem-2-cpu-times} denotes the time that is required for solving one local problem in the online phase where the average is taken over all nodes and all iteration steps.} The offline time is much higher (depending on the resolution of the coarse mesh) but can be further parallelized depending on the number of available CPUs. The more CPUs that we can use for distributing the computations, the cheaper the method.

$\\$
{{\bf Conclusion.} In this work we proposed a new method for tackling parametrized and nonlinear multiscale problems in an efficient way. We combined the LOD 
{method with a model reduction strategy to construct localized reduced basis functions.}
After an offline preprocessing step, these locally supported basis functions can be computed with very low computational costs for any new parameter. The space that is spanned by the functions is low dimensional and yields high approximation properties. The applicability of the RB-LOD was demonstrated in numerical experiments for linear and nonlinear problems.}

$\\$
{{\bf Acknowledgment.} This work was supported in part by the Swiss National Science Foundation under Grant 200021 134716. We also thank the anonymous referees for their helpful feedback which improved this article.

\def\cprime{$'$}

\end{document}